\documentclass[a4paper, 12pt]{amsart}

\usepackage[all,cmtip]{xy}
\usepackage{amssymb}
\usepackage{amsfonts}
\usepackage{amsthm}
\usepackage{amsmath}
\usepackage{amscd}
\usepackage{color} 
\usepackage{hyperref}
\usepackage{fullpage}

\usepackage{graphicx}
\usepackage[all]{xy}
\usepackage{array}
\usepackage{youngtab}

\newcommand{\abs}[1]{\lvert#1\rvert}

\newcommand{\li}{ { \,|\,   }  }

\newcommand{\bbC}{{\mathbb {C}}}

\newcommand{\bbG}{{\mathbb {G}}}

\newcommand{\bbN}{{\mathbb {N}}}

\newcommand{\bbQ}{{\mathbb {Q}}}

\newcommand{\bbZ}{{\mathbb {Z}}}

\newcommand{\sfA}{{\mathsf {A}}}

\newcommand{\sfF}{{\mathsf {F}}}
\newcommand{\sfG}{{\mathsf {G}}}
\newcommand{\sfH}{{\mathsf {H}}}

\newcommand{\sfL}{{\mathsf {L}}}

\newcommand{\sfS}{{\mathsf {S}}}
\newcommand{\sfT}{{\mathsf {T}}}
\newcommand{\sfU}{{\mathsf {U}}}

\newcommand{\sfX}{{\mathsf {X}}}
\newcommand{\sfY}{{\mathsf {Y}}}
\newcommand{\sfZ}{{\mathsf {Z}}}

\newcommand{\CF}{{\mathcal {F}}}

\newcommand{\CS}{{\mathcal {S}}}

\newcommand{\rmF}{{\mathrm {F}}}
\newcommand{\rmG}{{\mathrm {G}}}

\newcommand{\rmM}{{\mathrm {M}}}

\newcommand{\rmO}{{\mathrm {O}}}

\newcommand{\rmQ}{{\mathrm {Q}}}
\newcommand{\rmR}{{\mathrm {R}}}

\newcommand{\rmV}{{\mathrm {V}}}

\newcommand{\val}{{\mathrm{val}}}

\newcommand{\Hom}{{\mathrm{Hom}}}

\newcommand{\GL}{{\mathrm{GL}}}

\newcommand{\rank}{{\mathrm{rank}}}

\newcommand{\od}{\operatorname{d}}

\newcommand{\und}{\underline}
\newcommand{\incl}{\hookrightarrow}

\newcommand{\Z}{\mathbb{Z}}
\newcommand{\C}{\mathbb{C}}

\newcommand{\la}{\langle}
\newcommand{\ra}{\rangle}

\newcommand{\be}{\begin{equation}}
\newcommand{\ee}{\end{equation}}
\newcommand{\bt}{\begin{theorem}}
\newcommand{\et}{\end{theorem}}
\newcommand{\bd}{\begin{definition}}
\newcommand{\ed}{\end{definition}}
\newcommand{\bp}{\begin{proposition}}
\newcommand{\ep}{\end{proposition}}
\newcommand{\bl}{\begin{lemma}}
\newcommand{\el}{\end{lemma}}
\newcommand{\bco}{\begin{corollary}}
\newcommand{\eco}{\end{corollary}}
\newcommand{\br}{\begin{remark}}
\newcommand{\er}{\end{remark}}
\newcommand{\bex}{\begin{example}}
\newcommand{\eex}{\end{example}}
\newcommand{\ben}{\begin{enumerate}}
\newcommand{\een}{\end{enumerate}}
\newcommand{\bc}{\begin{cases}}
\newcommand{\ec}{\end{cases}}

\newcommand{\bpf}{\begin{proof}}
\newcommand{\epf}{\end{proof}}

\newcommand{\bma}{\begin{bmatrix}}
\newcommand{\ema}{\end{bmatrix}}

\theoremstyle{Theorem}

\theoremstyle{Theorem}

\theoremstyle{Theorem}

\theoremstyle{Definition}

\newtheorem{theorem}{Theorem}[section]

\newtheorem{lemma}[theorem]{Lemma}
\newtheorem{proposition}[theorem]{Proposition}
\newtheorem{corollary}[theorem]{Corollary}

\newtheorem{remark}[theorem]{Remark}

\theoremstyle{definition}
\newtheorem{definition}[theorem]{Definition}

\newtheorem{example}[theorem]{Example}

\begin{document}

\title{The poles of Igusa zeta integrals and the unextendability of semi-invariant distributions}

\author{Jiuzu Hong\\ With an appendix joint with Shachar Carmeli}
\address{Jiuzu Hong\\
Department of Mathematics, University of North Carolina at Chapel Hill, CB 3250 Phillips Hall Chapel Hill, NC 27599-3250, U.S.A.}
\email{jiuzu@email.unc.edu}
\address{ 
Shachar Carmeli\\
Faculty of Mathematics and Computer Science, Weizmann Institute of Science, POB 26, Rehovot 76100, Israel.}
\email{shachar.carmeli@gmail.com}


\maketitle
\begin{abstract}
We investigate the relationship between the poles of Igusa zeta integrals and the unextendability of semi-invariant distributions. Under some  algebraic conditions,  we  obtain an upper bound for the order of the poles of Igusa zeta integral, and by using the order of the poles we  give a criterion on the unextendability of  semi-invariant distributions.  A key ingredient of our method is  the idea of generalized semi-invariant distributions. 
\end{abstract}

%\tableofcontents

\section{Introduction}
Following Bernstein-Zelevinsky \cite{BZ}, we define an $\ell$-space
to be a topological space which is  Hausdorff, locally compact,
totally disconnected and second-countable.  An $\ell$-group is a
topological group whose underlying topological space is an
 $\ell$-space. Let $G$ be an $\ell$-group acting  on an $\ell$-space $X$.  Let $D(X)^{\chi}$ be the space of $\chi$-invariant distributions on $X$
 \[ D(X)^{\chi}:=\Hom_G(\CS(X),\chi) ,\]
where $\CS(X)$ is the space of Schwartz functions on $X$ and $G$ acts on $\CS(X)$ naturally. Here a character $\chi$ of $G$ is a continuous homomorphism $\chi: G\to  \bbC^\times$, where we take the discrete topology on $\bbC^\times$. 

  If $X$ is a $G$-homogeneous space, then by Frobenius reciprocity it is easy to determine the space $D(X)^{\chi}$.  It is either $1$-dimensional or $0$.  When it is nonzero, we say $\chi$ is $\bold{admissible}$ on $X$.    If $X$ is not a homogeneous space,  it is in general a  difficult question to determine the space $D(X)^{\chi}$.  Let $D(X)^{\chi,\infty}$ be the space of generalized $\chi$-invariant distributions on $X$ (see definition in Section \ref{hom_space}). Let $\Lambda_G$ be the quotient of $G$ by the normal subgroups generated by all compact open subgroups.   The action of $\Lambda_G$ on $D(X)^{\chi,\infty}$ determines the space $D(X)^{\chi}$ of $\chi$-invariant distributions on $X$.  This point of view originated in \cite{HS} and it continues to play important role in this paper.  
  
Let $\rmF$ be a non-archimedean local field of characteristic zero with the normalized norm $\abs{\cdot}$. It was proved  in Tate's thesis \cite{Ta} that for any character $\chi$ of $\rmF^\times$,  the space $D(\rmF)^{\chi}$ is of dimension $1$.  When $\chi$ is nontrivial, the proof is easy. When $\chi$ is trivial, the proof can be sketched as follows. 
We  attach a zeta integral 
\[ \int_{\rmF}  \phi(x)  \abs{x}^s  \frac{dx}{\abs{x}}, \]
for any Schwartz function $\phi$ on $\rmF$. It has a simple pole at $s=0$.  Taking all coefficients of the Laurent expansion of the zeta integral at $s=0$ and considering the natural action of $\rmF^\times$ on the space of generalized invariant distributions, we can conclude that the only invariant distribution on $\rmF$ is the delta distribution which is supported at the origin.

 %B.\,Sun and the author  consider  much  more general setting in \cite{HS}. We use the coefficients of Laurent expansion of Igusa zeta integral to produce generalized semi-invariant distributions and  prove that the semi-invariant distributions and more generally generalized semi-invariant invariant distributions can be extended as generalized semi-invariant distributions under some algebraic-geometric conditions.   In fact in \cite{HS} we also obtain an automatic extension theorem by using homological algebraic methods, which will be also crucial to the computation in Section \ref{ex_sect}.
 
  For  the applications of invariant distributions in representation theory of $p$-adic group and number theory,  it is  important to know whether a semi-invariant distribution can be  extended to a semi-invariant distribution on the whole space, for example see a recent work by Gourevitch-Sahi-Sayag \cite{GSS}. The work in \cite{HS} focuses on the question when generalized semi-invariant distributions can be extended. In this paper we pursue the investigation of the extension problem of semi-invariant distributions following the ideas and techniques of generalized semi-invariant distributions developed in \cite{HS}. Distinctively in this paper we focus on the unextendability of semi-invariant distributions, i.e. when a semi-invariant distribution can not be extended. 

Let $\sfX$ be an algebraic variety over $\rmF$ and let $f$ be a regular function on $\sfX$. Let $\mu$ be an algebraic measure on $\sfX_f(\rmF)$ where $\sfX_f$ is the open subvariety of $\sfX$ defined by $f$. We can attach a zeta integral $Z_{f,\mu}(\phi)$ for any Schwartz function on $\sfX(\rmF)$,
\[Z_{f,\mu}(\phi)=\int_{\sfX_f(\rmF) }    \phi(x)\abs{f}^s d\mu. \]
 
 In general $Z_{f,\mu}$  absolutely converges when $\Re(s)\gg 0$ and has meromorphic continuation.  We impose an action of a linear algebraic group $\sfG$ over $\rmF$ on $\sfX$ and assume $f$ and $\mu$ are semi-invariant.  We investigate the relationship between the poles of $Z_{f,\mu}$ and the unextendability of $\mu$ to $\sfX(\rmF)$ as a semi-invariant distribution.  The principle is that the order of the pole of zeta integral controls the unextendability of semi-invariant distribution.  We state the main result in Theorem \ref{Main}. In the main theorem, we impose some algebraic conditions on a stratification of $V(f)=\sfX(\rmF)\backslash \sfX_f(\rmF)$. We obtain a bound on the order of the pole, which relies on the stratification, and we also obtain an unextendability criterion based on the order of the pole.  The above example  on p-adic line in Tate thesis  serves as the most basic example.

   When the Schwartz function $\phi$ is the characteristic function of $\sfX(\rmR)$ where $\rmR$ is the ring of integers in $\rmF$, the monodromy conjecture by Igusa (cf.\cite{Ig1,De2} predicts that  the poles of Igusa zeta function $Z_{f,\mu}(\phi)$ are closely related to the b-function $b_f$ associated to $f$  (introduced by Bernstein\cite{Be1}). More precisely   if $s_0$ is a pole of $Z_{f,\mu}(\phi)$, then the real part $\Re(s_0)$ is a root of $b_f$.  In fact when $\sfX$ is a prehomogeneous space, it  has been well understood (cf.\cite{KSZ}),  and the b-functions for prehomogeneous spaces are also very computable (cf.\cite{Ki,SKKO}).   
Very often the orders of the poles and their locations are computable, and therefore in many cases one can determine whether the semi-invariant distributions can be extended or not. 

 %it is  predictable and verifiable on the unextendability of semi-invariant distributions by using the pole of Igusa zeta integral.   
  
  In Section \ref{main_thm}, we first define several notions including standard Igusa  zeta integrals and fiberizable spaces,  and then  we state the main theorem (Theorem \ref{Main}), and  give some corollaries (Corollary \ref{cor_1},\ref{cor_2} ).  Section \ref{proof_theorem} is devoted to the proof of Theorem \ref{Main}.  In Section \ref{hom_space}, we first review the basics of generalized semi-invariant distributions and then we determine the space of generalized semi-invariant distributions on algebraic homogeneous space in  Proposition \ref{key}.  In Section \ref{loc_section},  we use  equivariant $\ell$-sheaves  to prove a version of  localization principle for generalized semi-invariant distributions (Proposition \ref{loc_act}, \ref{act_fiber} ).   In Section \ref{non-deg-pairing}, we prove the existence of a non-degenerate pairing between the lattice of algebraic characters of a connected linear algebraic group $\sfG$ over $\rmF$ and the lattice $\Lambda_{\sfG(\rmF)}$ in Proposition \ref{non_deg}.   In Section  \ref{inv-section} we use the invariance of Igusa zeta integral to prove  Proposition \ref{residue}.   Combining all these preparations, the proof of the main theorem is concluded in Section \ref{final}.
  
  As an illustration of using Theorem \ref{Main} and its corollaries, we compute all semi-invariant distributions on the space $X=\rmF^n\times \rmF^n$ with an action of the group $G=\rmF^\times \times \GL_n\times \rmF^\times$. As a consequence we show that for any character $\chi$ of $G$, the space of $\chi$-invariant distributions on $X$ is at most 1-dimensional (Theorem \ref{multiplicity_one}).  There are different phenomenons when  $n=1$, $n=2 $ and $n\geq 3$. 
 When $n=1$, our main theorem is not  applicable when $\chi$ is trivial and we need  new arguments in Lemma \ref{case3lem}.  When $n=2$, the main result of this paper is very crucially used.

In  Appendix \ref{appendix}, we relate the residues of meromorphic intertwining operators to connecting maps in the Ext-groups long exact sequences. We give an explicit description for the first connecting map in terms of residues (Theorem \ref{Main_Theorem}).   Using the idea of generalized homomorphisms we  give a criterion for the triviality of the $1$-cocycles constructed from residues (Theorem \ref{Criterion_Non_Van}).   This is a homological algebra perspective on the unextendability of invariant distributions. 

{\bf Acknowledgements}: 
J.\,Hong would like to thank  B.\,Sun for the collaboration in \cite{HS} and many discussions on the computation of the examples in Section \ref{ex_sect}. He  would like to thank the hospitality of Weizmann institute of Science, where the part of the work was done during his visit in June  2017. 
 J.\,Hong is partially supported by the Simons Foundation collaboration grant 524406. 
 
 He would like to thank the anonymous referee for very careful reading and valuable comments which greatly improved the exposition of the paper.

   \section{Main theorem}
\label{main_thm}
Let $\rmF$ be a non-archimedean local field of characteristic zero and let $\rmR$ be the ring of  $p$-adic integers in $\rmF$ with a uniformizer $\pi$. Let $q$ be the cardinality of the residue field of $\rmF$ and let $\abs{\cdot }$ the normalized norm on $\rmF$, i.e.  $\abs{\pi}=q^{-1}$.

 As usual, by an algebraic variety over $\rmF$, we mean a scheme over $\rmF$ which is separated, reduced, and of finite type. A linear algebraic group over $\rmF$ is a group scheme over $\rmF$ which is an affine variety as a scheme.

 Let $\mathsf X$ be an  algebraic variety over $\rmF$. Let $\sfZ(f)$ be the subvariety of $\sfX$ defined by $f=0$.   Let $\sfU$ be the complement of $\sfZ(f)$ in $\sfX$.
 The set $\sfX(\rmF)$ of $\rmF$-rational points  in $\sfX$ is naturally an $\ell$-space.    Given a measure $\mu$ on $\sfU(\rmF)$ (in the sense of \cite[Definition 5.8]{HS}), for any $\phi\in \CS(\sfX(\rmF))$
we associate to it a zeta integral
$$Z_{f,\mu}(\phi):=\int_{\sfU(\rmF)} \phi(x) \abs{ f(x) } ^s d\mu .$$

\bd
We say that the zeta integral $Z_{f,\mu}$ is standard at $s=s_0$ if 
\ben
%\item  $\mu$ is a complex measure on $\sfU(\rmF)$ (cf. \cite[Definition 5.8]{HS}). 
\item for any $\phi\in \CS(\sfX(\rmF))$,  $Z_{f,\mu}(\phi)$ absolutely converges when $s\gg 0$, and $Z_{f,\mu}(\phi)$ admits a meromorphic continuation to the whole complex plane.
\item there exists an integer $n_0\in \bbN$ such that for any $\phi\in \CS(\sfX(\rmF))$,  $(s-s_0)^{n_0}Z_{f,\mu}(\phi)$ is analytic at $s=s_0$. 
\een
\ed
It is clear that $Z_{f,\mu}$ is standard at $s=s_0$ if and only if $Z_{f,\abs{f}^{s_0}\mu}$ is standard at $s=0$. 
If $Z_{f,\mu}$ is standard at $s=s_0$, we define the order of the pole of $Z_{f,\mu}$ at $s=s_0$ to be  $n_0$, if $(s-s_0)^{n_0}Z_{f,\mu}(\phi)$ is analytic for any $\phi\in \CS(\sfX(\rmF))$ and $(s-s_0)^{n_0-1}Z_{f,\mu}(\phi_0)$ has pole at $s=s_0$ for some $\phi_0\in \CS(\sfX(\rmF))$.  In particular the order of the pole of $Z_{f,\mu}$ at $s=s_0$ is always greater or equal than the order of the pole of $Z_{f,\mu}(\phi)$ at $s=s_0$ for any $\phi\in \CS(\sfX(\rmF))$.

 Igusa zeta integrals have been studied intensely  by Igusa\cite{Ig2}, Denef\cite{De} and many others. In general the zeta integral $Z_{f,\mu}(\phi)$ is a rational function in $q^{-s}$. In \cite[Theorem 5.13]{HS}  a general version of absolute convergence and meromorphic continuation in the setting of semi-algebraic $\ell$-space is proved. 
 \vspace{5pt}
  %They provide the following examples in \cite[Section 6.7]{HS}.
%\bex 
%Let $\sfX$ be a variety over $\rmF$. Let $f$ be a regular function on $\sfX$ and let $\sfX_f$ be the open subvariety of $\sfX$ defined by $f$. Let $\sfU$ be the smooth part of $\sfX_f$. 
%\ben
%\item  Assume that $ \sfX_f$ has rational singularities.  Let $\omega$ be an algebraic volume form on $\sfU$.   Then  $|\omega|$ extends to a measure on $\sfX_f(\rmF)$,  and the zeta integral $Z_{f,  \psi\abs{\omega}}$ is standard at $s=s_0$ for any $s_0\in \bbC$ where $\psi$ is any locally constant function on $\sfX(\rmF)$.
%\item  Assume that $\sfX_f$ has Gorenstein singularities.  Let   $\omega$ be an algebraic $n$-volume form on $\sfU$. Then  $|\omega|^{\frac{1}{n} }$ extends to a measure on  $\sfX_f(\rmF)$,  and the zeta integral $Z_{f,\psi\abs{\omega}^{\frac{1}{n}} }$  is standard at at $s=s_0$ for any $s_0\in \bbC$ where $\psi$ is any locally constant function on $\sfX(\rmF)$.
%\een
%\eex

Let $\sfG$ be  a connected linear algebraic group $\sfG$ over $\rmF$ acting on  the variety $\sfX$ over $\rmF$.  The  $\ell$-group $\sfG(\rmF)$ acts on the $\ell$-space $\sfX(\rmF)$ continuously.

An algebraic character $\nu$ of $\sfG$ over $\rmF$ is a homomorphism $\nu:  \sfG\to  \bbG_m$ of linear algebraic groups over $\rmF$ where $\bbG_m$ is the split multiplicative group over $\rmF$.   
 Let $\Psi_\sfG$ be the lattice of all algebraic characters of $\sfG$ over $\rmF$.  
 A regular function $f$ on $\sfX$ is $\nu$-invariant where $\nu$ is an algebraic character of $\sfG$ over $\rmF$, if 
\[ f(g\cdot x)=\nu(g)f(x)  \quad  \text{ for any }  g\in \sfG(\bar{\rmF}), x\in  \sfX(\bar{\rmF}) ,  \]
where $\bar{\rmF}$ is an algebraically closure of  $\rmF$.

 \begin{definition}
 We say an algebraic character $\nu$ of $\sfG$ is admissible on a homogeneous space $\sfG(\rmF)/\sfH(\rmF)$ for some algebraic subgroup $\sfH$ of $\sfG$ if the restriction of $\nu$ on $\sfH$ is trivial.
 \end{definition}
For any $\sfG(\rmF)$-orbit $\rmO$ in $\sfX(\rmF)$, we define $\Psi_{\rmO}$ to be a subgroup of $\Psi_\sfG$ consisting of all admissible algebraic characters of $\sfG$ on $\sfG(\rmF)\cdot x$, where $x\in \rmO$.  The group $\Psi_\rmO$ does not depend on the choice of $x\in \rmO$.

Given any algebraic character $\nu$ of $\sfG$ over $\rmF$, we denote by $\abs{\nu}$ the associated character of $\sfG(\rmF)$, i.e. $\abs{\nu}(g):=\abs{\nu(g)}$ for any $g\in \sfG(\rmF)$.

\bd
\label{fiberified}
Let $X$ be an $\ell$-space with an action of an $\ell$-group $G$.  We say $X$ is  $\bold{fiberizable}$ if there exists an $\ell$-space $Y$ and a morphism $h: X\to Y$ of $\ell$-spaces such that $h(g\cdot x)=h(x)$ for any $g\in G$, and for any $y\in Y$ the fiber $h^{-1}(y)$ is a disjoint union of finitely many closed  $G$-orbits. 
\ed

\br
Let $\sfX$ be a variety over $\rmF$ with an action of a linear algebraic group $\sfG$ over $\rmF$.  If $\sfX$ has a geometric quotient $\pi: \sfX\to \sfY$ (see the definition in \cite{MFK}), then $\sfX(\rmF)$ is  fiberizable.  The reason is that for any $y\in \sfY(\rmF)$,  the fiber $\psi^{-1}(y)$ consists of finitely many closed $\sfG(\rmF)$-orbits by the finiteness of Galois cohomology (see \cite[Section 6.4, Corollary 2 and Section 3.1, Corollary 2]{PR}).
\er

The main result of this paper is the following theorem.
\bt
\label{Main}
Let $\sfX$ be an algebraic variety over $\rmF$ with an action of a connected linear algebraic group $\sfG$ over $\rmF$. 
Let $f$ be a $\nu$-invariant regular function on $\sfX$ where  $\nu$ is  an algebraic character of $\sfG$ over $\rmF$.  Let  $V(f)$ be the zero set of $f$ on   $\sfX(\rmF)$ and put $U=\sfX(\rmF) \backslash V(f)$. 
Let  $\mu$ be a $\chi$-invariant measure on $U$ where $\chi$ is a character of $\sfG(\rmF)$.
Assume that $Z_{f,\mu} $ is standard at $s=0$, moreover assume that there exists a filtration $\O=V_{-1}\subset  V_0\subset V_1\subset  \cdots V_{k-1}\subset V_k=V(f)$ of closed $\sfG(\rmF)$-stable subsets of $V(f)$ such that
\ben
% \item  $\chi$ is trivial on $N$.
 \item   for each $i$,  all $\chi$-admissible orbits in $V_i\backslash V_{i-1}$ are contained in a fiberizable $G$-stable locally closed subset  of $V_i\backslash V_{i-1}$;
 \item for each $i$,  $\nu \not\in \Psi_i\otimes_{\bbZ}\bbQ$ where $\Psi_i$ is the sublattice of $\Psi_{\sfG}$ spanned by  $\Psi_\rmO$ for all $\chi$-admissible orbits $\rmO$ in $V_i\backslash V_{i-1}$. 
  \een
 
Then  $Z_{f,\mu}$ has pole at $s=0$ of order $\leq  \ell_\chi$, where 
\[ \ell_\chi:=\sharp  \{  i \li   V_i\backslash V_{i-1}  \textrm{  contains at least one  } \chi\text{-admissible orbit}     \} .\]
  Moreover, if $\ell_\chi \geq 1$ and  the order of the pole of  $Z_{f,\mu}$ at $s=0$ is exactly equal to $\ell_\chi$, then  $\mu$ can not be extended to  $\sfX(\rmF)$ as a $\chi$-invariant distribution.
\et
In the above theorem, if $Z_{f,\mu}$ is standard at $s=s_0$ and we replace $\chi$ by $\chi\abs{\nu}^{s_0}$ in the conditions, then $Z_{f,\mu}$ has a pole at $s=s_0$ of order $\leq \ell_\chi$ and if the order is exactly  $\ell_\chi$, then $\abs{f}^{s_0}\mu$ can not be extended to a $\abs{\nu}^{s_0}\chi$-invariant distribution on $\sfX(\rmF)$. 

 By Theorem \ref{Main}, the order of the pole of $Z_{f,\mu}$ at $s=0$ is bounded by $\ell_\chi$. If we can find a Schwartz function $\phi$ on $\sfX(\rmF)$ such that the order of the pole of $Z_{f,\mu}(\phi)$ at $s=0$ is equal to $\ell_\chi$, then  the order of the pole of $Z_{f,\mu}$ at $s=0$ is equal to $\ell_\chi$.

Theorem \ref{Main} will be proved in Section \ref{proof_theorem}.  We first give some remarks on the conditions in Theorem \ref{Main}.
 
\br
\ben
\item  Given any variety $\sfX$ defined over $\rmF$ with an action of a linear algebraic group $\sfG$ defined over $\rmF$, by a theorem of Rosenlicht ( see \cite{Ro}, \cite[Theorem 4.4, p.187]{PV} ),  there always exists a  filtration of $\sfG$-stable closed subvarieties $\sfX_0\subset \sfX_1\subset \cdots \subset \sfX_k=\sfX$ such that for each $i$, $\sfX_i\backslash \sfX_{i-1}$ has geometric quotient. Hence for each $i$,  $\sfX_i(\rmF)\backslash  \sfX_{i-1}(\rmF)$ is fiberizable.
\item  For each $i$, if $V_i\backslash V_{i-1}$ contains  finitely many $\chi$-admissible orbits and there is no closure relation among these orbits, then Condition (1) of  the theorem is satisfied.  
\item  Let $\rmO$ be a $\chi$-admissible $\sfG(\rmF)$-orbit in $\sfX(\rmF)$.    Let $\Xi_\rmO$ be the set of all admissible characters of $\sfG(\rmF)$ on $\rmO$. Then $\Psi_\rmO$ acts on $\Xi_\rmO$, i.e. for any $\chi\in \Xi_\rmO$,  $|\beta|^{s}\chi\in \Xi_\rmO$ for any $\beta\in \Psi_\rmO$ and $s\in \bbC$. In particular if the set $\Xi_\rmO$  consists of only countablely many admissible characters of $\rmG(\rmF)$, then $\Psi_\rmO$ is trivial.  
\een
\er

The following corollary is an immediate consequence of Theorem \ref{Main}.
\bco
\label{cor_1}
With the same setup as in Theorem \ref{Main},
assume that $Z_{f,\mu} $ is standard at $s=0$, moreover we assume that 
\ben
% \item  $\chi$ is trivial on $N$.
 \item   all $\chi$-admissible orbits in $V(f)$ are contained in  $Y$ where $Y$ is a fiberizable $\sfG(\rmF)$-stable  locally closed subset of $V(f)$.
 
 \item  $\nu \not\in \Psi_{V(f)} \otimes_{\bbZ}\bbQ$ where $\Psi_{V(f)}$ is the sublattice of $\Psi_\sfG$ spanned by  $\Psi_\rmO$ for all $\chi$-admissible orbits $\rmO$ in $V(f)$. 
  \een
 Then $Z_{f,\mu}$ has at most a simple pole at $s=0$, and $Z_{f,\mu}$ has a pole at $s=0$  if and only if $\mu$ can not be extended to a $\chi$-invariant distribution on $\sfX(\rmF)$.
\eco

\bex

We consider the following action of $\rmG:={\rm GL}_n(\rmF)\times {\rm GL}_n(\rmF)$ ($n\geq 1$) on the  space ${\rm M}_{n}:={\rm M}_{n,n}(\rmF)$ of $n\times n$-matrices with coefficients in $\rmF$:
$$(g_1,g_2)\cdot x:=g_1 x g_2^{-1}, \qquad \text{ } g_1\in {\rm GL}_n(\rmF),\, g_2\in {\rm GL}_n(\rmF), \,x\in {\rm M}_{n}. $$
Let $\rm O_r$ denote the set of rank $r$ matrices in ${\rm M}_{n,n}$, which is a $\rmG$-orbit. The matrix space $\rmM_{n}$ is union of $\rmO_r, r=0,1,\cdots, n$.
Let $\od \!x$ denote the Haar measure on $\rmM_{n}$  and for any character $\chi$ of $\rmF^\times$, we associate the following zeta integral  
\[  Z_{\det, \chi(\det)\!\od\!x } (\phi)=\int_{\rmM_{n}}\phi(x) \abs{\det}^s \chi(\det)\! \od\!x ,  \quad  \text{ for any } \phi\in \CS(\rmM_{n}).  \]
It is well-known that $Z_{\det, \chi(\det)\!\od\!x } $ is standard everywhere.  The Haar measure $\od\!x$ is $(\abs{\cdot}^{n},\abs{\cdot}^{-n})$-invariant,  and $\chi(\det)\!\od\!x $ is $(\chi\abs{\cdot}^{n}, \chi^{-1}\abs{\cdot}^{-n})$-invariant, where we denote any character of $\rmG$ by $(\chi_1,\chi_2)$ for some characters $\chi_1,\chi_2$ of $\rmF^\times$ via
\[(g_1,g_2)\mapsto  \chi_1(\det(g_1))\chi_2(\det(g_2)), \quad \text{ for any } g_1,g_2\in \GL_n.\]
  The computation (cf.\cite[Chapter 10.1]{Ig})
$$\int_{{\rm M}_n(\rmR)}  |{\rm det}(x)|^s \od\!x=\prod_{i=1}^n  \frac{1-q^{-i}}{1-q^{-i-s}} $$
shows that when $\chi=\abs{\det}^{-i}, i=1,2,\cdots, n$, $Z_{\det, \chi(\det)\!\od\!x }$ has pole at $s=0$. Note that when $r<n$, $\rmO_r$ is the only  $(\abs{\cdot}^{r}, \abs{\cdot}^{-r})$-admissible orbit and the lattice $\Psi_{\rmO_r}$ of admissible algebraic characters on $\rmO_r$ is trivial.  Corollary \ref{cor_1} immediately implies that for any $r=0,1,\cdots, n-1$,  $Z_{\det, \abs{\det}^{r-n} \!\od\!x }$ has simple pole at $s=0$ and the $(\abs{\cdot}^{r}, \abs{\cdot}^{-r})$-invariant distribution  $\abs{\det}^{r-n} \!\od\!x $ on  $\rmO_n$ can not be extended to $\rmM_{n}$ as a $(\abs{\cdot}^{r}, \abs{\cdot}^{-r})$-invariant distribution.   The only $(\abs{\cdot}^{r}, \abs{\cdot}^{-r})$-invariant distribution on $\rmM_{n}$ is obtained from the residue of $Z_{\det, \abs{\det}^{r-n} \!\od\!x }$ at $s=0$. For any $\chi\neq \abs{\cdot }^r$, the $(\chi,\chi^{-1})$-invariant distribution on $\rmM_{n}$ are obtained from the extension of $\chi(\det) \abs{\det}^{-n}  \!\od\!x $ on $\rmO_n$ via analytic continuation.
Therefore $\dim D(\rmM_n)^{\chi,\chi^{-1}}=1$ for any character $\chi$ of $\rmF^\times$.  In particular Tate's thesis on the p-adic line is a special case.   This example has been computed in \cite[Section 7]{HS} by using generalized semi-invariant distributions.  With the help of Theorem \ref{Main} and Corollary \ref{cor_1}, the arguments can be greatly simplified.  
 \eex
We  emphasize that the lattice $\Psi_i$ in Theorem \ref{Main} or $\Psi_{V(f)}$ in Corollary \ref{cor_1} might not be trivial, see the example in Proposition \ref{case_2} in Section \ref{ex_sect}.

From Theorem \ref{Main}, we can get many more corollaries by adjusting the conditions. The following is another example.

\bco
\label{cor_2}
With the same setup as in Theorem \ref{Main},  assume that $Z_{f,\mu} $ is standard at $s=0$. Moreover, we assume that   
\ben
 \item there are finitely many $\chi$-admissible orbits in $V(f)$,
 \item     $\nu\not \in \Psi_{\rmO}\otimes_\bbZ \bbQ $ for any $\chi$-admissible orbit $\rmO$ in $V(f)$.
 \een  
Then the order of the pole of $Z_{f,\mu}$ at $s=0$ is  bounded by the number of $\chi$-admissible orbits in $V(f)$.

\eco
\bpf
First of all we note that any $\sfG(\rmF)$-orbit $\rmO$ is locally closed.  We may label all $\chi$-admissible orbits in $V(f)$ as $\rmO_1,\rmO_2,\cdots, \rmO_k$, such that for each $1\leq i\leq k$, $\rmO_i$ is not contained in the closure  of any orbit of $\rmO_1,\cdots, \rmO_{i-1}$.  Set $V_k=V(f)$, $V_1=\bar{\rmO}_1$ and for any $2\leq r\leq k-1$, set 
\[V_{r}= (\bar{\rmO}_{r+1}\backslash \rmO_{r+1})\cup ( \cup_{i=1}^{r} \bar{\rmO}_i),\]
where $\bar{\rmO}_i$ is the closure of the orbit $\rmO_i$ for each $i$.  It gives a filtration of $\sfG(\rmF)$-stable closed subsets $\emptyset =V_0\subset V_1\subset \cdots  \subset V_k=V(f)$ where for each $1\leq i\leq k$,  $V_i\backslash V_{i-1}$  contains exactly one $\chi$-admissible orbit $\rmO_i$.  Therefore the corollary follows from Theorem \ref{Main}.  \epf

\section{Proof of main Theorem }
\label{proof_theorem}
   In this section, we are devoted to prove  Theorem \ref{Main}. 
   
   \subsection{Generalized semi-invariant distributions on homogeneous spaces}
   \label{hom_space}
   
Given an $\ell$-space $X$ with the action of an $\ell$-group $G$, let $S(X)$ be the space of Schwartz functions on $X$, i.e. locally constant $\bbC$-valued function with compact support on $X$.  We define the action of $G$ on $S(X)$ as follows,
 \[(g\cdot \phi)(x)=\phi(g^{-1}\cdot x),\quad   \text{  for any } g\in G, \, \phi\in S(X), \text{ and  } x\in X.\]  
 It gives a left action of $G$ on $S(X)$.
 Let $D(X)$ denote the space of distributions on $X$, i.e. all linear functionals on $S(X)$.  We define the $\bold{right}$ action of $G$ on $D(X)$ via
 \[  (\xi\cdot g)(\phi)=\xi(g\cdot \phi),\quad  \text{ for any } g\in G,\, \xi\in D(X) \text{ and } \phi\in S(X)  .\]

Given a character $\chi$ of $G$, 
we denote by $D(X)^{\chi,k}$ the space consisting of distributions $\xi$ on $X$ such that 
\[  \xi\cdot (g_1-\chi(g_0))(g_1-\chi(g_1))\cdots (g_k-\chi(g_k))=0,\quad   \text{for any } g_0,g_1,\cdots, g_k\in G .   \]

Put
\[D(X)^{\chi,\infty}:= \bigcup_k D(X)^{\chi,k}
\]
Any distirbution $\xi$ in $D(X)^{\chi,\infty}$ is called a generalized $\chi$-invariant distribution on $X$, and any  distribution $\xi\in D(X)^{\chi,k}$ is called a generalized $\chi$-invariant distribution of order $\leq k$.  For any $k\in \bbN$,  the space $D(X)^{\chi, k}$ still carries the action of $G$. Any generalized $\chi$-invariant distribution of order $\leq 0$ is equivalent to be $\chi$-invariant.
 For any $g\in G$, the operator $g-\chi(g)$ acts on $D(X)^{\chi,\infty}$ nilpotently.

Let $G^\circ$ be the subgroup of $G$ generated by all open compact subgroups of $G$.   $G^\circ$ is nomral in $G$.
We put
   \[ \Lambda_G:=G/G^\circ. \]

 Let $J_{G,k}=\bbC[\Lambda_G]/  (I_G)^{k+1}$, where $I_G$ is the augmentation ideal of $\bbC[\Lambda_G]$, i.e. 
 \[  I_G:=\{ \sum_{g\in \Lambda_G}  a_g g\in \bbC[\Lambda_G]|   \sum_{g\in \Lambda_G} a_g =0     \}. \]

Note that $J_{G,k}$ carries a natural action of $G\times G$. The following lemma follows from  \cite[Lemma 2.6]{HS}.
\bl 
\label{Lem_2.6}
$D(X)^{\chi,k}=\Hom_G(S(X)\otimes J_{G,k}, \chi)$, where $G$ acts on $J_{G,k}$ from the left.
\el

\bl
\label{comp}
Given a representation $V$ of a compact group $K$, any generalized $\chi$-invariant vector is $\chi$-invariant. 
\el
\bpf
It follows from the complete reducibility of representation of compact group.
\epf

In the rest of this subsection, we will determine all generalized semi-invariant distributions on algebraic homogeneous spaces.

Let $\sfG$ be a connected linear algebraic group over $\rmF$. Let $\sfH$ be an algebraic subgroup of $\sfG$.   In the rest of this subsection, we always use $G$ to denote $\sfG(\rmF)$ and use $H$ to denote $\sfH(\rmF)$. 

The following lemma is well-known (cf. \cite[Prop.7.2.1]{Ig}).
\bl
\label{adm_lem}
For any character $\chi$ of $G$, the homogeneous space
  $G/H$ is $\chi$-admissible if and only if
   \[
\chi|_{H}=\abs{\Delta_\sfG} \cdot \abs{\Delta_{\sfH}} ^{-1},
\]
where the algebraic modular character $\Delta_{\sfG}$ is given by  the $1$-dimensional representation of $\sfG$ on $\wedge^{\rm top} \mathfrak{g}$, where $\mathfrak{g}$ is the Lie algebra of $\sfG$ with the adjoint action of $\sfG$, and the algebraic character $\Delta_{\sfH}$ is defined similarly.
\el

It is clear that $D(G/H)^{\chi,\infty}\neq 0$ if and only if $\chi$ is admissible on $G/H$. 
When the character $\chi$ of $G$ is trivial on the unipotent radical of $G$,
the following proposition is an immediate consequence of  \cite[Theorem 6.15]{HS}.   In fact the condition on the character $\chi$ can be removed. We give the proof for general case here.

\bp
\label{key}
Assume that $G$ is connected and that $X=G/H$ is $\chi$-admissible. Then any $\xi\in D(X)^{\chi,\infty}$ can be written as
\[   
\xi= P(\val\circ \alpha_1, \val\circ \alpha_2,\cdots, \val \circ \alpha_r) \mu,
\]
where $P$ is a polynomial in $r$ variables, $\alpha_1,\alpha_2,\cdots, \alpha_r\in \Psi_{G/H}$ ( recall that $\Psi_{G/H}$ is the group of admissible algebraic characters of $\sfG$ on $G/H$),   and $\mu$ is the $\chi$-invariant distribution on $X$. 
\ep

\bpf

Let $\xi$ be any generalized $\chi$-invariant distribution on $X$ of order $\leq k$.  Choose an open compact subgroup $K$ of $G$.  We decompose $X$ as the union of $K$-orbits, 
\[ X=\sqcup_{i} X_i .\]
Then the distribution $\xi|_{X_i}$ on $X_i$ is a generalized $\chi|_{K}$-invariant. In view of Lemma \ref{comp}, it is automatically $\chi|_K$-invariant.   Hence there exists a constant number $\alpha_i$ such that $\xi|_{X_i}=\alpha_i \mu|_{X_i}$.  It follows that $\xi=f\cdot \mu$, where $f:X\to \bbC$ is a locally constant function on $X$ which is generalized invariant of order $\leq k$ with respect to the action of $G$.  

By pulling back to $G$, $f$ can be viewed  as a generalized invariant locally constant function on $G$ which is trivial on $H$.  In view of Lemma \ref{comp}, $f$ can further descend to a generalized invariant function on $\Lambda_G$ which is trivial with respect to the action of $H$.  By \cite[Prop. 6.8, Lemma 6.12]{HS}, 
$f$ can be written as a polynomial in 
$\val\circ \alpha_1, \val\circ \alpha_2,\cdots, \val \circ \alpha_r$ of degree $\leq k$, for some algebraic characters $\alpha_1,\alpha_2,\cdots, \alpha_r$ of $G$ which are trivial on $H$.  It finishes the proof of the proposition. 
\epf

\bco
\label{hom_action}
Assume that $X=G/H$ is $\chi$-admissible.  Let $g$ be an element in $G$.  If for any algebraic character $\nu\in \Psi_X$, $|\nu(g)|=1$, then $g$ acts on $D(X)^{\chi,\infty}$ by $\chi(g)$.
\eco
\bpf

If $\chi$ is not admissible on $X$, then $D(X)^{\chi,\infty}=0$.  The corollary trivially holds. 
Otherwise, it is an immediate consequence of  Proposition \ref{key}.

\epf

\subsection{A localization principle for generalized semi-invariant distributions}
  \label{loc_section} 
Let $X$ be an $\ell$-space. We define an $\ell$-sheaf on $X$ to be a
sheaf of complex vector spaces on $X$. For any $\ell$-sheaf $\CF$ on
$X$, let $\Gamma_c(X,\CF)$ denote the space of all  global sections of
$\CF$ with compact support. In particular,
$S(X)=\Gamma_c(X,\C_X)$, where $\C_X$ denotes the sheaf of
locally constant $\C$-valued functions on $X$. For each $x\in X$,
denote by $\CF_x$ the stalk of $\CF$ at $x$; and for each $s\in
\Gamma_c(X,\CF)$, denote by $s_x\in \CF_x$ the germ of $s$ at $x$.
The  set $\bigsqcup_{x\in X} \CF_x$ carries a unique topology  such
that for all $s\in \Gamma_c(X,\CF)$, the map
  \[
    X\rightarrow \bigsqcup_{x\in X} \CF_x, \quad x\mapsto s_x
  \]
is an open embedding. Then $\Gamma_c(X,\CF)$ is naturally identified
with the space of all compactly supported continuous sections of the
map $\bigsqcup_{x\in X} \CF_x\rightarrow X$.

Let $G$ be an $\ell$-group acting continuously on an
$\ell$-space $X$.

\begin{definition}( \cite[Section 1.17]{BZ})
A $G$-equivariant $\ell$-sheaf on $X$ is an $\ell$-sheaf $\CF$ on
$X$, together with a continuous group action
\[
  G\times \bigsqcup_{x\in X} \CF_x\rightarrow \bigsqcup_{x\in X} \CF_x
\]
such that for all $x\in X$, the action of each $g\in G$ restricts to a
linear map $\CF_x\to \CF_{g.x}$.
\end{definition}

 Given a $G$-equivariant $\ell$-sheaf $\CF$ on $X$, the space $\Gamma_c(X,\CF)$ is a smooth representation of $G$ so that
\[
  (g.s)_{g.x}=g.s_x \quad \textrm{for all } g\in G,\, x\in X,  \,s\in \Gamma_c(X,\CF).
\]
For each $G$-stable locally closed subset $Z$ of $X$, the
restriction $\CF|_Z$ is clearly a  $G$-equivariant $\ell$-sheaf on
$Z$.

We define the space of distributions on $\CF$ as the dual of $\Gamma_c(X,\CF)$, i.e. 
 \[ D(X,\CF):=\Gamma_c(X,\CF)^* . \]
   Similar to the right action of $G$ on $D(X)$,  we have a right action of $G$ on $D(X,\CF)$. For any character $\chi$ of $G$. Let $D(X,\CF)^{\chi,k}$ be the space of generalized $\chi$-invariant distributions of order $\leq k$, for $k=0,1,\cdots$, and put
   \[ D(X,\CF)^{\chi, \infty}=\bigcup_{k=0}^\infty D(X,\CF)^{\chi,k} .\] 
   For each $k$, $D(X,\CF)^{\chi,k}$ carries a right  action of $G$ which is locally finite.

Recall the $G\times G$-module  $J_{G,k}$ in Section \ref{hom_space}. Similar to Lemma \ref{Lem_2.6}, we have the following lemma.
\bl For any $k$, 
\label{a_Lem_2.6}
$D(X,\CF)^{\chi,k}=\Hom_G(\Gamma_c(X,\CF)\otimes J_{G,k}, \chi)$, where   $G$ acts on  $J_{G,k}$ from the left.
\el

\bl
\label{loc_lemma}
Let $X$ be an $\ell$-space and let $G$ be an $\ell$-group acting trivially on $X$. Fix elements $g_1,g_2,\cdots, g_n$ in $G$. For any $G$-equivariant $\ell$-sheaf $\CF$ on $X$, if $(g_1-1)(g_2-1)\cdots (g_n-1)$ acts on $\CF_x$ by zero for any $x\in X$, then $(g_1-1)(g_2-1)\cdots (g_n-1)$ also acts by zero on $\Gamma_c(X,\CF)$.
\el
\bpf
The lemma is immediate, since
the space $\Gamma_c(\CF)$ can be identified with all compactly supported continuous sections of the map  $\bigsqcup_{x\in X} \CF_x\rightarrow X$.
\epf

\bp
\label{loc_act}
Let  $p: X\to Y$ be a $G$-equivariant morphism of $\ell$-spaces such that $G$ acts on $Y$ trivially.  Let $\CF$ be a $G$-equivariant $\ell$-sheaf  $\CF$ on $X$. For any character $\chi$ of $G$ and elements $g_1,g_2,\cdots,g_n\in G$,   if  $(g_1-\chi(g_1))(g_2-\chi(g_2))\cdots (g_n-\chi(g_n))$ acts on $D(p^{-1}(y), \CF )^{\chi,\infty}$  by zero for any $y\in Y$, then  $(g_1-\chi(g_1))(g_2-\chi(g_2))\cdots (g_n-\chi(g_n))$ acts on $D(X, \CF)^{\chi,\infty}$  also by zero.

\ep
\bpf
 It suffices to show that for any $k$,  the action of $(g_1-\chi(g_1))(g_2-\chi(g_2))\cdots (g_n-\chi(g_n))$ on $D(X,\CF)^{\chi,k}$ is zero. 

First of all in view of Lemma \ref{a_Lem_2.6},  we have the following natural isomorphisms of $G$-modules
\[  D(X,\CF)^{\chi,k}=(\Gamma_c(X,\CF_{\chi,k})_G)^*,\]
and
\[   D(p^{-1}(y),\CF|_{p^{-1}(y)})^{\chi,k}=(\Gamma_c(p^{-1}(y),  \CF_{\chi,k}|_{p^{-1}(y)})_G)^*  \]
for any $y\in Y$ and $k=0,1,\cdots$.
    Here  $\CF_{\chi,k}:=\CF \otimes \chi^{-1} \otimes J_{G,k}$ is naturally a $G\times G$-equivariant sheaf, where the first copy of $G$ acts diagonally on $\CF$,  $\chi^{-1}$   and  $J_{G,k}$  from the left, and the second copy of $G$ acts individually on the right of $J_{G,k}$.   

By a theorem of Bernstein-Zelevinsky (cf.\cite[Proposition 2.36]{BZ}),  there exists an $\ell$-sheaf $(\CF_{\chi,k})_G$ on $Y$ such that
   \[  \Gamma_c(Y, (\CF_{\chi,k})_G)= \Gamma_c(X,   \CF_{\chi,k} )_G ,  \]
  and for any $y\in Y$,
  \[ ( (\CF_{\chi,k})_G) _y =\Gamma_c(p^{-1}(y),   \CF_{\chi,k}|_{p^{-1}(y) })_G, \]
  where the coinvariants are taken with respect to the diagonal action of $G$. 
  
 The $\ell$-sheaf $ (\CF_{\chi,k})_G$ still carries the action of the second copy of $G$, which acts on $Y$ trivially.   Moreover the right $G$-actions on $D(X,\CF)^{\chi,k}$ and  $D(p^{-1}(y), \CF|_{p^{-1}(y)} )^{\chi,k}$ exactly comes from the action of the second copy of $G$ on $(\CF_{\chi,k})_G$.  By Lemma \ref{loc_lemma},  the proposition follows.  
\epf

The following is a version of localization principle on generalized invariant distributions. 
\bp
\label{act_fiber}
Let $X$ be an $\ell$-space with action of $G$. Let $\pi: X\to Y$ be a $G$-equivariant continuous map of $\ell$-spaces, where $G$ acts on $Y$ trivially.  Assume that for any $y\in Y$,  $\pi^{-1}(y)$ is a disjoint union of finitely many closed $G$-orbits. Given a character $\chi$ of $G$ and elements $g_1,g_2,\cdots, g_n\in G$, if  $(g_1-\chi(g_1))(g_2-\chi(g_2))\cdots (g_n-\chi(g_n))$ acts on $D(\rmO)^{\chi,\infty}$ by zero for any orbit $\rmO$ in $X$, then $(g_1-\chi(g_1))(g_2-\chi(g_2))\cdots (g_n-\chi(g_n))$ acts on $D(X)^{\chi,\infty}$ also by zero.
\ep
\bpf
For any $y\in Y$,  since $\pi^{-1}(y)$ is a disjoint union of finitely many closed $G$-orbits, we have
 \[D(\pi^{-1}(y))^{\chi,\infty} =\bigoplus_{\rmO} D(\rmO)^{\chi,\infty}, \]
 where the summation is taken over $G$-orbits in $\pi^{-1}(y)$. 
 By assumption $(g_1-\chi(g_1))(g_2-\chi(g_2))\cdots (g_n-\chi(g_n))$ acts on $D(\pi^{-1}(y))^{\chi,\infty}$ by zero. Now one can  easily see that this proposition follows from Proposition \ref{loc_act}.
\epf

We recall the following localization principle of Bernstein-Zelevinsky on the vanishing of invariant distributions, which will be used throughout the paper. 
\begin{theorem}(\cite[Theorem 6.9]{BZ} )
\label{Localization}
Suppose that the action $\gamma$ of an $\ell$-group $G$ on an $\ell$-space $X$ is constructible, i.e. the set $\{(x,g\cdot x  )\,|\,   g\in G, x\in X  \}$ is a union of finite many locally closed subsets of $X\times X$. 
% Let $\CF$ be a $G$-equivariant $\ell$-sheaf on $X$. If there are no nonzero $G$-invariant $\CF$-distributions on any $G$-orbit in $X$, then there are no non-zero $G$-invariant $\CF$-distributions on $X$. As a consequence, 
Given a character $\chi$ of $G$, if there are no nonzero $\chi$-invariant distributions on any $G$-orbit in $X$, then there are no non-zero $\chi$-invariant distributions on $X$.
\end{theorem} 
   
   \begin{remark}
   \label{Constructibility}
\begin{enumerate}
\item Given an algebraic action of a linear algebraic group  $\sfG$ over $\rmF$ on an algebraic variety $\sfX$ over $\rmF$,  the induced action of $\sfG(\rmF)$ on $\sfX(\rmF)$ is constructible (cf. \cite[Theorem 6.15]{BZ}). 
\item  If an action of $\ell$-group $G$ on an $\ell$-space $X$ is constructible, then on any $G$-stable locally closed subset $Z$ of $X$, the action of $G$  remains constructible. 
 \end{enumerate}  
  \end{remark}
\subsection{A non-degenerate pairing of lattices}
   \label{non-deg-pairing}
   
Let $\sfG$ be a connected linear algebraic group defined over  $\rmF$. 
 Let
\[
  \Psi_ \sfG:=\Hom( \sfG,  \bbG_m),
\]
be the group of all algebraic characters on $ \sfG$.
Here $\bbG_m$ denotes the multiplicative group of $\rmF$. Then $\Psi_\sfG$ is a lattice, i.e. a finitely generated free abelian group. Write $\mathrm e_\sfG$ for its rank.  When $\sfG$ is reductive, ${\rm e}_\sfG$ is equal to the rank of the split central torus. 

Define a  map
\be \label{pg}
  \sfG(\rmF)\times \Psi_\sfG\rightarrow \Z,\qquad (g, \nu)\mapsto \val(\nu(g)),
\ee
where $\val$ is the valuation map on the non-archimedian field $\rmF$.  Recall that $\Lambda_{\sfG(\rmF)}$ is the quotient of $\sfG(\rmF)$ by $\sfG(\rmF)^\circ$ the subgroup of $\sfG(\rmF)$ generated by all compact open subgroups.

\bl
The map \eqref{pg} descends to a bilinear map 
\[
 \la\,,\,\ra_\sfG:  \Lambda_{\sfG(\rmF)}\times \Psi_\sfG\rightarrow \Z.
\]
\el
\bpf
It suffices to prove that, for any algebraic character $\nu\in \Psi_{\sfG}$ and for any $g\in \sfG(\rmF)^\circ$ we always have $\val(\nu(g))=0$.

The map  $\val\circ \nu:  \sfG(\rmF)\to  \bbZ$ is continous, where we take discrete topology on $\bbZ$.  For any open compact subgroup $K$ of $\sfG(\rmF)$,  the image is a compact subgroup of $\bbZ$, which is forced to be trivial.  Hence the lemma follows.  
\epf

The paring $\la\,,\,\ra_\sfG$ is compatible with  homomorphisms of algebraic groups in the sense of the following obvious lemma:
\bl\label{paring}
Let $\phi: \sfG\rightarrow \sfH$ be a homomorphism of connected linear algebraic groups defined over $\rmF$. Then
\[
  \la g, \nu\circ \phi\ra_\sfG=\la \Lambda_\phi(g), \nu\ra_\sfH
\]
for all $g\in \Lambda_{\sfG(\rmF)}$ and $\nu\in \Psi_\sfH$. Here $\Lambda_\phi:\Lambda_{\sfG(\rmF)}\rightarrow \Lambda_{\sfH(\rmF)}$ denotes the homomorphism induced by $\phi: \sfG\rightarrow \sfH$.
\el

Let $\sfL$ be a connected reductive group defined over $\rmF$. 
We write
\[
  \sfL=\sfS \sfA [\sfL, \sfL],
\]
where $\sfS$ denotes the maximal split central torus in $\sfL$, $\sfA$ denotes the maximal anisotropic central torus in $\sfL$, and $[\sfL, \sfL]$ denotes the derived subgroup of $\sfL$. Write
\[
  \sfL^\dagger:=\sfA [\sfL, \sfL].
\]

\bl\label{lcirc}
We have $\sfL^\dagger(\rmF)^\circ=\sfL^\dagger(\rmF)$.
\el
\begin{proof}
Since $[\sfL, \sfL]$ is a semisimple connected  linear algebraic group over $\rmF$, every open normal subgroup of $[\sfL, \sfL](\rmF)$ has finite index in it (cf \cite[Proposition 3.18]{PR}). Since $[\sfL,\sfL](\rmF)^\circ$ is open and normal in $[\sfL,\sfL](\rmF)$, it implies that 
\[
   ([\sfL, \sfL](\rmF))^\circ=[\sfL, \sfL](\rmF).
\]
Since $\sfA$ is an anisotropic torus,  $\sfA(\rmF)$ is compact. Therefore $(\sfA(\rmF))^\circ=\sfA(\rmF)$. Now we have that
\[
  (\sfL^\dagger(\rmF))^\circ\supset ([\sfL, \sfL](\rmF))^\circ (\sfA(\rmF))^\circ=([\sfL, \sfL](\rmF))(\sfA(\rmF)).
\]
Note that $([\sfL, \sfL](\rmF))(\sfA(\rmF))$ has finite index in $\sfL^\dagger(\rmF)$ (cf. \cite[Corollary 2 of Theorem 6.16]{PR}). It follows that $\sfL^\dagger(\rmF)^\circ=\sfL^\dagger(\rmF)$. 
\end{proof}

The following lemma was stated in \cite[Chapter II, Prop.22]{Be}.
\bl\label{freeab2}
The group $\Lambda_{\sfL(\rmF)}$ is a lattice of rank $\mathrm e_\sfL$.
\el
\begin{proof}
Note that $\sfL(\rmF)/\sfL^\dagger(\rmF)$ is topologically isomorphic to a finite index subgroup of $(\sfL/\sfL^\dagger)(\rmF)$  (cf \cite[Corollary 2 of Theorem 6.16]{PR}).   It implies that $\Lambda_{\sfL(\rmF)/\sfL^\dagger(\rmF)}$ is a sublattice of $\Lambda_{(\sfL/\sfL^\dagger)(\rmF)}$ with torsion quotient. Hence the rank of $\Lambda_{\sfL(\rmF)/\sfL^\dagger(\rmF)}$ is  $e_{\sfL/\sfL^\dagger}=e_\sfL$.

Now Lemma \ref{lcirc} implies that
 $ \Lambda_{\sfL(\rmF)}=\Lambda_{\sfL(\rmF)/\sfL^\dagger(\rmF)}.$
This proves the lemma.

\end{proof}

The following lemma is obvious.
\bl\label{injtori}
\ben
\item  Let $\sfT$ is a split torus. Then $\la\,,\,\ra_\sfT$ is a perfect paring, namely it induces an isomorphism from $\Lambda_{\sfT(\rmF)}$ to $\Hom(\Psi_\sfT, \Z)$.

\item Let $\phi: \sfT_1\rightarrow \sfT_2$ be a surjective homomorphism of split tori over $\rmF$ with finite kernel. Then its induced homomorphism $\Lambda_{\sfT_1(\rmF)}\rightarrow \Lambda_{\sfT_2(\rmF)}$ is injective.
\een
\el

We are now ready to prove the following proposition.   
\bp\label{non_deg}
The paring $\la\,,\,\ra_\sfG: \Lambda_{\sfG(\rmF)}\times \Psi_{\sfG}\to \bbZ$ is non-degenerate. 
\ep

\bpf

Let $\sfL$ be a Levi component of $\sfG$, namely it is an algebraic subgroup of $\sfG$ such that $\sfG=\sfL\ltimes \sfU_\sfG$, where $\sfU_\sfG$ denotes the unipotent radical of $\sfG$. Then the projection homomorphism $\sfG\rightarrow \sfL$ induces an identification
\[
  \Psi_\sfG=\Psi_\sfL.
\]
Recall that $\sfU_\sfG(\rmF)$ is the union of all its compact subgroups. Therefore the projection homomorphism $\sfG\rightarrow \sfL$ also induces an identification
\[
  \Lambda_{\sfG(\rmF)}=\Lambda_{\sfL(\rmF)}.
\]
In view of Lemma \ref{paring}, it suffices to prove the proposition for the connected reductive  group $\sfL$.   

As above we write
\[
  \sfL=\sfS \sfA [\sfL, \sfL],
\]
where $\sfS$ denotes the maximal split central torus in $\sfL$, $\sfA$ denotes the maximal anisotropic central torus in $\sfL$, and $[\sfL, \sfL]$ denotes the derived subgroup of $\sfL$. Write
\[
  \sfL^\dagger:=\sfA [\sfL, \sfL].
\]

The natural algebraic group homomorphisms
\be \label{sgg}
\sfS\rightarrow \sfL\rightarrow \sfL/\sfL^\dagger
\ee
induce group homomorphisms
\be \label{sgg2}
  \Lambda_{\sfS(\rmF)}\rightarrow \Lambda_{\sfL(\rmF)}\rightarrow \Lambda_{(\sfL/\sfL^\dagger)(\rmF)}.
\ee
The homomorphisms \eqref{sgg} also induce group homomorphisms
\be \label{sgg3}
\Psi_\sfS\leftarrow \Psi_\sfL\leftarrow \Psi_{\sfL/\sfL^\dagger}.
\ee
It is clear that the two homomorphisms in \eqref{sgg3} are injective, and three lattices have the same rank.

Part (2) of Lemma \ref{injtori} implies that $\Lambda_{\sfS(\sfF)}\to  \Lambda_{(\sfL/\sfL^\dagger)(\rmF)}$ is injective.  Hence the map $\Lambda_{\sfS(\sfF)}\to  \Lambda_{\sfL(\rmF)}$ is also injective.  In view of Lemma \ref{freeab2} and part (1) of Lemma \ref{injtori}, the lattice  $ \Lambda_{\sfL(\rmF)}$ has the same rank as $\Lambda_{\sfS(\sfF)}$.  It implies that $ \Lambda_{\sfL(\rmF)}\rightarrow \Lambda_{(\sfL/\sfL^\dagger)(\rmF)}$ is also injective.

In view of Lemma \ref{paring}, the non-degeneracy of $\la\,,\,\ra_\sfS$ and $\la\,,\,\ra_{\sfL/\sfL^\dagger}$ imply the non-degeneracy of the  paring $\la\,,\,\ra_\sfL$. Consequently, the  paring $\la\,,\,\ra_\sfG$ is also non-degenerate.

\epf

 \subsection{Invariance of Igusa zeta integral}
\label{inv-section}
In this subsection we keep the same setup as in Theorem \ref{Main}.  Recall that $f$ is $\nu$-invariant, i.e.  $f(g\cdot x)=\nu(g)f(x)$,  and $\mu$ is $\chi$-invariant. The zeta integral $Z_{f,\mu}$ satisfies the following invariance:
\be
\label{inv}
 Z_{f,\mu} \cdot g= |\nu(g)|^s \chi(g)  Z_{f,\mu} .  
 \ee

For any $\phi\in \CS(\sfX(\rmF))$,  $Z_{f,\mu}(\phi)$ is a meromorphic function in $s$.   
Consider the Laurent expansion of $Z_{f,\mu}$ 
\[Z_{f,\mu}=\sum   Z_{f,\mu,i}s^i ,  \]
where $Z_{f,\mu,i}\in  D(\sfX(\rmF))^{\chi,\infty}$ for each $i\in \bbZ$.    Let $i_0$ be the largest integer such that $Z_{f,\mu,-i_0}\neq 0$ and $Z_{f,\mu,-i_0-1}=0$.

\bl
\label{igu_inv}
\ben
\item  For any $i$, we have the following relation  \[ Z_{f,\mu, i}\cdot (g-\chi(g))=\chi(g)\sum_{k=1 }^\infty  \frac{ (\log|\nu(g)|)^k}{k!}   Z_{f,\mu, i-k} ,  \]    
%\item For any $i\geq -i_0$,  $Z_{f,\mu, i}\neq 0$.   
\item The coefficients $Z_{f,\mu, -i_0}, Z_{f,\mu, -i_0+1},\cdots,  Z_{f,\mu,0},\cdots$ are linearly independent in $D(\sfX(\rmF))^{\chi,
\infty}$.
\een
\el
\bpf
The Taylor expansion of $|\nu(g)|^s$ at $s=0$ is 
\[ |\nu(g)|^s =\sum_{j=0}^\infty   \frac{(\log|\nu(g)|)^j}{j!} s^j ,\]
By comparing  the Laurent expansion of  the both sides of (\ref{inv}), we have 
\be
\label{recursive}
 Z_{f,\mu,i}\cdot g=  \chi(g)(Z_{f,\mu,i}  + (\log|\nu(g)| )Z_{f,\mu,i-1}+  \frac{(\log|\nu(g)|)^2 }{2}Z_{f,\mu, i-2} +\cdots   ). 
 \ee
It proves the part (1) of the lemma. 

Note that $Z_{f,\mu, -i_0}\neq 0$ is $\chi$-invariant.  Use the formula (\ref{recursive}) and by induction we can easily see that 
for any $i\geq -i_0$, $Z_{f,\mu, i}\neq 0$ and $Z_{f,\mu, i}\in D(\sfX(\rmF))^{\chi, i+i_0}$.

Part (2) of the lemma also easily follows from  the formula (\ref{recursive}). 
\epf

\bp
\label{residue}
If the zeta integral  $Z_{f,\mu}$ has pole at $s=0$, then there is at least one $\chi$-admissible orbit in $V(f)$. 
\ep

\bpf
If there is no $\chi$-invariant distributions on every $G$-orbit in $V(f)$,  by Theorem \ref{Localization} there is no $\chi$-invariant distribution on $V(f)$.  

If $Z_{f,\mu}$ has pole at $s=0$ of order $r$, in view of Lemma \ref{igu_inv}  $Z_{f,\mu,-r}$ is a $\chi$-invariant distribution and supported in $V(f)$.  It is a contradiction.  Hence there exists at least one $\chi$-admissible orbit in $V(f)$.
\epf
\br  This proposition is a generalization of a result of Igusa,  Gyoja in the case of group action on vector spaces (cf. \cite[Theorem 8.5.1]{Ig} and the remark therein). 
\er

It implies that if there is no $\chi$-admissible orbit in $V(f)$,  then $Z_{f,\mu}$ is analytic at $s=0$, hence $\mu$ can be extended to a $\chi$-invariant distribution on $\sfX(\rmF)$. Comparing with \cite[Theorem 1.4]{HS}, in special cases we get stronger result, i.e. we don't need to assume non-weakly admissibility. 
\subsection{Final proof}
\label{final}

Before we  conclude Theorem \ref{Main}, we need to make more preparations. 
\bl
\label{incl_lem}
Let $X$ be an $\ell$-space with a constructible action of an $\ell$-group $G$. Let $Y$ be a $G$-stable locally closed subset of $X$. let $\chi$ be a character of $G$. If all $\chi$-admissible orbits in $X$ are contained in $Y$, then we have  a natural embedding of $G$-modules
\[ D(X)^{\chi,\infty}\incl  D(Y)^{\chi,\infty} . \]
\el

\bpf
Let $\overline{Y}$ be closure of $Y$ in $X$.  Note that we have the following exact sequence of $G$-modules
\[  0\rightarrow  D(\overline{Y})^{\chi,\infty}\rightarrow D(X)^{\chi,\infty}\rightarrow D(X\backslash \overline{Y})^{\chi,\infty} .  \]

Since there is no $\chi$-admissible orbit in $X\backslash \overline{Y}$, by Theorem \ref{Localization} and part 2) of Remark \ref{Constructibility}, $D(X\backslash \overline{Y})^{\chi}=0$. It follows that 
$D(X\backslash \overline{Y})^{\chi,\infty}=0$.  Hence we have the isomorphism
\[    D(\overline{Y})^{\chi,\infty}\simeq D(X)^{\chi,\infty} . \]

Since $Y$ is open in $\overline{Y}$, we have the following exact sequence 
\[  0\rightarrow  D(\overline{Y}\backslash Y )^{\chi,\infty}\rightarrow D(\overline{Y})^{\chi,\infty}\rightarrow D(Y)^{\chi,\infty} .  \]

Using Theorem \ref{Localization} again, $D(\overline{Y}\backslash Y)^{\chi,\infty}=0$. Hence we get the embedding 
\[  D(X)^{\chi,\infty}\simeq   D(\overline{Y})^{\chi,\infty} \incl D(Y)^{\chi,\infty} . \]

\epf

Recall that $\Psi_i$ is the subgroup of $\Psi_{\sfG}$ spanned by  $\Psi_{\rmO}$ for all $\chi$-admissible orbits in $V_i\backslash V_{i-1}$.

   \bl
   \label{elem_lem}
   For any algebraic character $\nu\in \Psi_\sfG$, if $\nu\not\in \Psi_i\otimes_\bbZ \bbQ $, then there exits $g\in \sfG(\rmF)$ such that $|\nu(g)|\not= 1$ and for any $\nu'\in \Psi_i$, $|\nu'(g)|=1$.
   \el
   \bpf
      Let $\tilde{\Psi}_i$ be the lattice $\Psi_i +\bbZ \nu\subset \Psi_\sfG$. The  assumption $\nu\not\in \Psi_i\otimes_\bbZ \bbQ$ implies that
    \[\rank(\tilde{\Psi}_i)= \rank(\Psi_i)+1 . \]
    Let $\Psi_i^\perp$ be the following sublattice of $\Lambda_{\sfG(\rmF)}$ 
\[
     \Psi_i^\perp:=\{ \bar{g}\in \Lambda_{\sfG(\rmF)}  \li   \la \bar{g}, \Psi_i\ra_\sfG =0    \}  
     \]
    and similarly let $\tilde{\Psi}_i^\perp$ be the following lattice
        \[ \tilde{\Psi}_i^\perp:=\{ \bar{g}\in \Lambda_{\sfG(\rmF)}  \li   \la \bar{g}, \tilde{\Psi}_i\ra_\sfG =0    \}.  \]
 By the non-degeneracy of the pairing $\la ,\ra_\sfG: \Lambda_{\sfG(F)}\times  \Psi_{\sfG}\to \bbZ$ (Proposition \ref{non_deg}),  we have
 \[ \rank (\Psi_i^\perp)=\rank(\tilde{\Psi}_i^\perp)+1.  \]
    In particular there exists $g\in \sfG(\rmF)$ such that $\la \bar{g}, \nu \ra_\sfG\neq 0$ and $\la \bar{g}, \Psi_i \ra_\sfG= 0$.
    
       \epf
    By Lemma \ref{elem_lem},  for each $i$ there exists an element $g_{i} \in \sfG(\rmF)$ such that $|\nu(g_{i}  )|\neq 1$ and for any 
 $\nu' \in \Psi_{i} $, $\abs{\nu'(g_{i})}=1$.   In view of  Proposition \ref{key} and Corollary \ref{hom_action},  $g_i$ acts on $D(\rmO)^{\chi,\infty}$ by $\chi(g_i)$ for any orbit $\rmO$ in $V_i\backslash V_{i-1}$.
       \bl
       \label{strata_action}
     With the choice of $g_i$ as above, the operator  $g_i$ acts on $D(V_i\backslash V_{i-1})^{\chi,\infty}$ by the scalar $\chi(g_i)$.
       \el
       \bpf
       By assumption of Theorem \ref{Main}, there exists a fiberizable $\sfG(\rmF)$-stable locally closed subset $Y_i$ of $V_i\backslash V_{i-1}$ such that all $\chi$-admissible orbits in $V_i\backslash V_{i-1}$ are contained in $Y_i$.  By Proposition \ref{act_fiber}, $g_i$ acts on $D(Y_i)^{\chi,
       \infty}$ by $\chi(g_i)$. In view of Remark \ref{Constructibility}, the action of $\sfG(\rmF)$ on $V_i\backslash V_{i-1}$ is constructible. By Lemma \ref{incl_lem}, we have the embedding of $G$-modules $D(V_i\backslash V_{i-1})^{\chi,\infty}\subset D(Y_i)^{\chi,\infty}$. Hence $g_i$ acts on $D(V_i\backslash V_{i-1})^{\chi,\infty}$  by $\chi(g_i)$.
       \epf
       
Let $i_1<i_2<\cdots<  i_{\ell_\chi}$ be all integers such that  $V_{i_t}\backslash V_{i_t-1}$ contain at least one  $\chi$-admissible orbit. 
Let $g_{i_t} \in \sfG(\rmF)$ be the element such that $|\nu(g_{i_t}  )|\neq 1$ and for any 
 $\nu' \in \Psi_{i_t} $, $|\nu'(g_{i_t})|=1$.   
 
\bl
\label{residue_action}With the choice of elements $g_{i_1},g_{i_2},\cdots, g_{i_{\ell_\chi}}$ as above, 
the operator $(g_{i_1}-\chi(g_{i_1}))(g_{i_2}-\chi(g_{i_2}))\cdots  (g_{i_{\ell_\chi}}-\chi(g_{i_{\ell_\chi}}))$ acts on $D(V(f))^{\chi,\infty}$ by zero. 
\el

\bpf
First of all we look at the following exact sequence of $\sfG(\rmF)$-modules
 \[  0\to  D(V_{i_{\ell_\chi}})^{\chi,\infty} \to  D(V(f))^{\chi,\infty}\to  D(V(f) \backslash V_{i_{\ell_\chi}}  )^{\chi,\infty} . \]
 Since there is no $\chi$-admissible orbit in $V(f) \backslash V_{i_{\ell_\chi}} $, by Theorem \ref{Localization} $D(V(f) \backslash V_{i_{\ell_\chi}} )^\chi=0$, and hence $D(V(f) \backslash V_{i_{\ell_\chi}}  )^{\chi,\infty}=0$.
  It follows that  
  \[ D(V(f))^{\chi,\infty}\simeq D(V_{i_{\ell_\chi}})^{\chi,\infty}.\]

We use induction to show  that for any $t$,
$ (g_{i_1}-\chi(g_{i_1})) (g_{i_2}-\chi(g_{i_2}))\cdots  (g_{i_t}-\chi(g_{i_t}))$ acts on $D(V_{i_t})^{\chi,\infty}$ by zero.  When $t=1$, by  Lemma \ref{strata_action}  $g_{i_1}-\chi(g_{i_1})$ acts on $V_{i_1}$ by zero. 
Look at the following exact sequence 
\[  0\to  D(V_{i_{t-1}})^{\chi,\infty} \to  D(V_{i_{t}})^{\chi,\infty}\to  D(V_{i_t} \backslash V_{i_{t-1}})^{\chi,\infty} . \]
It suffices to show that $ (g_{i_1}-\chi(g_{i_1})) (g_{i_2}-\chi(g_{i_2}))\cdots  (g_{i_t}-\chi(g_{i_t}))$ acts on $D(V_{i_{t-1}})^{\chi,\infty} $ and $D(V_{i_t} \backslash V_{i_{t-1}})^{\chi,\infty}$ by zero simultaneously. 

By induction $  (g_{i_1}-\chi(g_{i_1})) (g_{i_2}-\chi(g_{i_2}) )\cdots  (g_{i_{t-1}}-\chi(g_{i_{t-1}})) $ acts on $D(V_{i_{t-1} })^{\chi,\infty}$ by zero, and by Lemma \ref{strata_action},  $g_{i_t}-\chi(g_{i_t})$ acts on $D(V_{i_t} \backslash V_{i_t-1} )^{\chi,\infty}$ by zero.  It follows  that $ (g_{i_1}-\chi(g_{i_1}))(g_{i_2}-\chi(g_{i_2}))\cdots  (g_{i_t}-\chi(g_{i_t}))$ acts on $D(V_{i_t})^{\chi,\infty}$ by zero.   This finishes the proof of the lemma. 
  \epf

Finally we come back to the proof of Theorem \ref{Main}. If  there is no $\chi$-admissible orbits in $V(f)$, in view of Proposition \ref{residue},  the theorem holds.  

From now on we assume that there is at least one $\chi$-admissible orbit in $V(f)$.  In this case $\ell_\chi\geq 1$.
 We first show that the pole of $Z_{f,\mu}$ at $s=0$ is bounded by $\ell_\chi$.  If $Z_{f,\mu}$ is analytic at $s=0$, then it is clearly true.  If $Z_{f,\mu}$ has pole at $s=0$, then $Z_{f,\mu,-1}\neq 0$ and it is supported on $V(f)$.   In view of Lemma \ref{igu_inv} and Lemma \ref{residue_action}, we have 
\begin{align}
& 0=Z_{f,\mu,-1}\cdot (g_{i_1}-\chi(g_{i_1}))(g_{i_2}-\chi(g_{i_2}))\cdots  (g_{i_{\ell_\chi}}-\chi(g_{i_{\ell_\chi}}))\\
 &=  \chi(g_{i_1} g_{i_2}\cdots g_{i_{\ell_\chi}})  \log\abs{\nu(g_{i_1})}\log\abs{\nu(g_{i_2})}\cdots \log\abs{\nu(g_{i_{\ell_\chi}} )  } \cdot Z_{f,\mu,-\ell_\chi-1}+\cdots,
\end{align}
where $\cdots$ denotes more other terms consisting of $Z_{f,\mu,i}$ which are nonzero when $i<-\ell_\chi-1$.  For each $t$, $ |\nu(g_{i_t})|\neq 1 \iff  \log(\abs{\nu(g_{i_t})})\neq 0$.  By the linear independence of $Z_{f,\mu,i}$ (Lemma \ref{igu_inv}),  $Z_{f,\mu,-\ell_\chi-1}=0$.   In particular it follows that the order of the pole of $Z_{f,\mu}$ at $s=0$ is bounded by $\ell_\chi$.\\

We now prove the second part of the Theorem \ref{Main}.  By assumption $Z_{f,\mu}$  has pole  of order $\ell_\chi$ at $s=0$ where $\ell_\chi\geq 1$.  
   Assume $\mu$ can be extended to a $\chi$-invariant distribution $\tilde{\mu}$.  Then $Z_{f,\mu,0}-\tilde{\mu}$ is supported in $V(f)$.  By Lemma \ref{residue_action},  
\be
\label{z_zero}
(Z_{f,\mu,0}-\tilde{\mu})\cdot (g_{i_1}-\chi(g_{i_1}))(g_{i_2}-\chi(g_{i_2}))\cdots  (g_{i_{\ell_\chi}}-\chi(g_{i_{\ell_\chi}})) =0.   
\ee

     On the other hand, in view of Lemma \ref{igu_inv}  we have 
     \begin{align}
 & (Z_{f,\mu,0}-\tilde{\mu})\cdot (g_{i_1}-\chi(g_{i_1}))(g_{i_2}-\chi(g_{i_2}))\cdots  (g_{i_{\ell_\chi}}-\chi(g_{i_{\ell_\chi}})) \\
&= Z_{f,\mu,0}\cdot (g_{i_1}-\chi(g_{i_1}))(g_{i_2}-\chi(g_{i_2}))\cdots  (g_{i_{\ell_\chi}}-\chi(g_{i_{\ell_\chi}}))   \\
  &=   \chi(g_{i_1} g_{i_2}\cdots g_{i_{\ell_\chi}})  \log\abs{\nu(g_{i_1})}\log\abs{\nu(g_{i_2})}\cdots \log\abs{\nu(g_{i_{\ell_\chi}})}  \cdot Z_{f,\mu,-\ell_\chi}
\end{align}
is not zero, since for each $t$, $\log(\abs{\nu(g_{i_t})})\neq 0$.
It  contradicts with (\ref{z_zero}).   Therefore $\mu$ can not be extended to a $\chi$-invariant distribution on $\sfX(\rmF)$.
This finishes the proof of Theorem \ref{Main}.

\section {Examples}

\label{ex_sect}

We consider the  action of $ \rmG= \rmF^\times \times \GL_n(\rmF)\times \rmF^\times$ on ${\rm V}:=\rmF^n\times \rmF^n$  given by  $$(a,g,b)\cdot (x,y)=(axg^{-1}, b^{-1}yg^t),$$
where $x=(x_1,x_2,\cdots,x_n), y=(y_1,y_2,\cdots,y_n)\in \rmF^n$, $a,b\in\rmF^\times $, $g\in \GL_n(\rmF)$ and $g^t$ is the transpose of $g$. 

For convenience we denote by $\und{\chi}=(\chi_1,\chi_2,\chi_3)$ the following character of $\rmG$,
\[(a,g,b)\mapsto  \chi_1(a)\chi_2(\det g) \chi_3(b). \]
Any character  of $\rmG$ is given by $(\chi_1,\chi_2,\chi_3)$ for some characters $\chi_1,\chi_2,\chi_3$ of $\rmF^\times$.

  In this section, we  prove the following theorem.    
\bt
\label{multiplicity_one}
For any character $\und{\chi}$ of $\rmG$, $ \dim D(\rmV)^{\und{\chi}}\leq  1$ .
\et
When $n=1$, $n=2$ and $n\geq3$, the arguments will be different. We will prove the more precise  statements in   Proposition \ref{case1}, Proposition \ref{case_2} and Proposition \ref{case_3} for each case.

The $\rmG$-space $\rmV$ is a  prehomogeneous space with the semi-invariant polynomial
\[ f(x,y)=x\cdot y^t=\sum_{i=1}^{n} x_iy_i.\]
   The polynomial $f$ is $\nu$-invariant, where $\nu$ is an algebraic character of $\rmG$ given by  $\nu(a,g,b)=ab^{-1}$, for any $(a,g,b)\in \rmG$.

Let $ \rmV_f$ be the subset of $\rmV$ such that $f\not=0$.  Let ${\rm Q}$ be the zero set of $f$.   We denote by ${\rm Q}^\times$ the  set
$   \{ (x,y)\in {\rm Q}|   x\not=0,y\not=0    \}  $.   Denote by ${\rm O}_{1,0}$ the subset $(\rmF^n\backslash\{0\})\times \{0\}\subset \rmV$ and $\rm O_{0,1}$ the subset $\{0\}\times( \rmF^n\backslash \{0\})\subset  \rmV$.
When $n=1$,  ${\rm Q}^\times $ is empty.
The following lemma is clear.
\begin{lemma}
\begin{enumerate}
\item If $n=1$,  the action of $\rmG$ on $\rmV$ consists of orbits $ \rmV_f,  {\rm O}_{1,0}, {\rm O}_{0,1},  \{0\}$.
\item  If $n\geq 2$,  the action of $\rmG$ on $\rmV$ consists of orbits $\rmV_f,  {\rm Q}^\times,  {\rm O}_{1,0}, {\rm O}_{0,1}, \{0\}$.
\end{enumerate}
\end{lemma}
The closure relation  of these orbits is  clear. 

The following lemma follows from Lemma \ref{adm_lem} by analysis on each orbit. 
\begin{lemma}
\label{basic_lemma_ex2}
Let $\und{\chi}=(\chi_1,\chi_2,\chi_3)$ be a character of $\rmG$.  Then
\begin{enumerate}
\item   $\und{\chi}$ is admissible on $\rmV_f$ if and only if
$$\begin{cases}
\chi_1\chi_2\chi_3=1   \qquad   \text{ if }  n=1\\
\chi_1\chi_3=1, \chi_2=1  \qquad \text{ if } n\geq 2.
\end{cases} $$
\item    $\und{\chi}$ is admissible on $\rmQ^\times$ if and only if
$$\begin{cases}
\chi_1\chi_2=\abs{\cdot},  \chi_2\chi_3=\abs{\cdot}^{-1}   \qquad   \text{ if }  n=2\\
\chi_1=\abs{\cdot}^{n-1},  \chi_2=1, \chi_3=\abs{\cdot}^{-n+1}  \qquad \text{ if } n\geq 3.
\end{cases} $$
\item  $\und{\chi}$ is admissible on ${\rm O}_{1,0}$ if and only if
$$\begin{cases}
\chi_1\chi_2=1,\chi_3=1   \qquad   \text{ if } \quad  n=1\\
\chi_1=\abs{\cdot}^{n}, \chi_2=\abs{\cdot }^{-1},  \chi_3=1   \qquad \text{ if } n\geq 2.
\end{cases} $$
\item  $\und{\chi}$ is admissible on ${\rm O}_{0,1}$ if and only if
$$\begin{cases}
\chi_1=1,  \chi_2\chi_3=1   \qquad   \text{ if }  n=1\\
\chi_1=1, \chi_2=\abs{\cdot},  \chi_3=\abs{\cdot}^{-n}   \qquad \text{ if } n\geq 2.
\end{cases} $$
\item  $\und{\chi}$ is admissible on $\{0\}$ if and only if $\chi_1=\chi_2=\chi_3=1$.
\end{enumerate}
\end{lemma}

We consider the following zeta integral
$$Z_{f,\chi(f)\mu}(\phi)=  \int_{\rmV}   \phi(v)  \abs{f(v)}^s  \chi(f(v))  d\mu, $$
where $\chi$ is a character of $\rmF^\times$,  $\mu$ is the Haar measure on $\rmV$ and $\phi\in \CS(\rmV)$. The zeta integral  $Z_{f,\chi(f)\mu}$ absolutely converges when $\Re(s)\gg 0$, moreover it admits a meromorphic continuation.
 
Note that $\mu$ is $(\abs{\cdot}^n, 1, \abs{\cdot}^{-n})$-invariant, and  $Z_{f,\chi(f)\mu}$ is $(\chi \abs{\cdot}^{s+n}, 1, \chi^{-1}\abs{\cdot}^{-s-n})$-invariant.

By the following computation (cf.\cite[Chapter 10.1]{Ig}),
  \be
  \label{Igusa_ex2}
  \int_{\rmR^{2n}} \abs{f(v)}^sd\mu=\frac{(1-q^{-1})(1-q^{-n})}{(1-q^{-1-s})(1-q^{-n-s})} ,
  \ee
 $Z_{f,\abs{f}^{-1}\mu}$  and $Z_{f, \abs{f}^{-n} \mu }$ has pole at $s=0$.

\bp
\label{case1}
Assume that $n\geq 3$.   Then $\dim D(\rmV)^{\und{\chi}}=1$ if and only if  
\be
\label{3_cond1}
\chi_1\chi_3=1\, , \chi_2=1 \quad \text{ or } 
\ee
  \be
  \label{3_cond2}
  (\chi_1,\chi_2,\chi_3)=(\abs{\cdot}^{n},  \abs{\cdot}^{-1}, 1)\, \text{ or } \, (1, \abs{\cdot}, \abs{\cdot}^{-n}).\ee
\ep
\bpf
When the conditions (\ref{3_cond1}) and (\ref{3_cond2}) do not hold, from Lemma \ref{basic_lemma_ex2} we see that any orbit in $V$ is not $\und{\chi}$-admissible. By Theorem \ref{Localization}, we have $D(\rmV)^{\und{\chi}}=0$.

Any semi-invariant distribution on $\rmV_f$ is given by $\chi(f)\mu$, which is $(\chi \abs{\cdot}^n, 1, \chi^{-1} \abs{\cdot}^{-n})$-invariant.   When $\chi\not=\abs{\cdot}^{-1}, \abs{\cdot}^{-n}$, the character  $(\chi \abs{\cdot}^n, 1, \chi^{-1} \abs{\cdot}^{-n})$ is not admissible on any orbit in $\rmQ$. By Proposition \ref{residue},  $Z_{f,\chi(f)\mu}$ is analytic at $s=0$. Hence  $\chi(f)\mu$  can be extended to a $(\chi\abs{\cdot}^n, 1, \chi\abs{\cdot}^{-n})$-invariant distribution on $\rmV$.  The extension is unique since other orbits are not $(\chi\abs{\cdot}^n, 1, \chi\abs{\cdot}^{-n})$-admissible.    

When $\chi=\abs{\cdot}^{-1}$ or  $\chi=\abs{\cdot}^{-n}$, the only possible $(\chi \abs{\cdot}^n, 1, \chi^{-1} \abs{\cdot}^{-n})$-admissible orbits in $\rmQ$ are  $\rmQ^\times$ and $\{0\}$. The lattice $\Psi_\rmO$ is trivial on  $\rmQ^\times$ and $\{0\}$. 
In view of (\ref{Igusa_ex2}), $Z_{f,\chi(f)\mu}$ has pole at $s=0$.  By  Theorem \ref{Main} or Corollary \ref{cor_1},  $Z_{f,\chi(f)\mu}$ has simple pole at $s=0$
and  $\chi(f)\mu$ can not be extended to  a $(\chi \abs{\cdot}^{n}, 1,\chi\abs{\cdot}^{-n})$-invariant distribution on $\rmV$.   In this case,  
$(1,1,1)$-invariant distribution on $\rmV$ is only supported on $\{0\}$, and the $(\abs{\cdot}^{1-n}, 1,\abs{\cdot}^{n-1})$-invariant distribution is contributed by  the $(\abs{\cdot}^{1-n}, 1,\abs{\cdot}^{n-1})$-invariant distribution on the closure of $\rmQ^\times$.

When $\und{\chi}=(\abs{\cdot}^{n},  \abs{\cdot}^{-1}, 1)\, \text{ or } \, (1, \abs{\cdot}, \abs{\cdot}^{-n})$,  by \cite[Theorem 1.4]{HS} $D(\rmV)^{\und{\chi}}=1$.   
It finishes the proof.
\epf

Now we consider the case $n=2$.  We denote by ${\rm Q}_{1,0}$ the set $\{(x,y)|f(x,y)=0, x\not=0\}$ and ${\rm Q}_{0,1}$ the set $\{(x,y)|f(x,y)=0,  y \not=0\}$.  Note that  $\{{\rm Q}_{1,0}, {\rm Q}_{0,1}\}$ gives an open covering of ${\rm Q}\backslash\{0\}$,  ${\rm Q}^\times ={\rm Q}_{1,0}\cap {\rm Q}_{0,1}$,  ${\rm Q}_{1,0}={\rm Q}^\times \cup {\rm O}_{1,0}$ and ${\rm Q}_{0,1}={\rm Q}^\times \cup {\rm O}_{0,1}$.

\begin{lemma}
\label{case2lemma}
\ben
\item If $\und{\chi}=(\abs{\cdot}^2,\abs{\cdot}^{-1}, 1)$, then $\dim D(\rmQ_{1,0})^{\und{\chi}}=1$.
\item   If $\und{\chi}=(1,\abs{\cdot},  \abs{\cdot}^{-2})$, then $\dim D(\rmQ_{0,1})^{\und{\chi}}=1$.
\een
\end{lemma}
\begin{proof}
We will only prove the part (1), as the proof for part (2) is similar.

Consider the projection map $p:  {\rm Q}_{1,0}\to  \rmF^2\backslash \{0\}$, given by $p(x_1,x_2,y_1,y_2)=(x_1,x_2)$.  The pre-image $p^{-1}(1,0)$ is $\{(1,0,0,y_2)| y_1\in \rmF \}\simeq \rmF$.  The stabilizer $H$ of $G$ at $(1,0)$ is
$$H=\{(a,\begin{bmatrix}a & 0  \\   g_{21} & g_{22} \end{bmatrix}, b)\li  \quad a,g_{21}, g_{22},b\in \rmF \}.$$
The action of $H$ on the fiber $p^{-1}(1,0)\simeq \rmF$ is given by
$$(a,\begin{bmatrix}a  & 0  \\   g_{21} & g_{22} \end{bmatrix}, b)\cdot y_2=b^{-1}g_{22}y_2. $$

By Frobenius reciprocity (cf.\cite[Section 1.5]{Be2}), there exists a natural isomorphism
$$D(\rmF)^{H}\simeq D({\rm Q}_{1,0})^{(\abs{\cdot}^2, \abs{\cdot}^{-1}, 1) }.$$
By Tate's thesis,  $\dim D(\rmF)^{H}=1$. It follows that  $\dim D({\rm Q}_{1,0})^{\abs{\cdot}^2, \abs{\cdot}^{-1}, 1) }=1$.
\end{proof}

\begin{proposition}
\label{case_2}
Assume that $n= 2$.   Then 
$\dim D(\rmV)^{\und{\chi}}=1$
 if and only if  
\be
\label{cond_1}
\chi_1\chi_3=1\, , \chi_2=1 \quad \text{ or } 
\ee
  \be
  \label{cond_2}
   \chi_1\chi_2=\abs{\cdot}, \chi_2\chi_3=\abs{\cdot}^{-1}    . 
   \ee
\end{proposition}
\bpf
When the conditions (\ref{cond_1}) and $(\ref{cond_2})$ do not hold,  Lemma \ref{basic_lemma_ex2} and Theorem \ref{Localization} implies that $D(\rmV)^{\und{\chi}}=0$.

By automatic extension theorem (cf.\cite[Theorem 1.4]{HS}),  it is easy to check that  if 
\[\und{\chi}\neq  (1,1,1), (\abs{\cdot}, 1, \abs{\cdot}^{-1}   ),  (\abs{\cdot}^2, \abs{\cdot}^{-1},  ),  (1, \abs{\cdot},  \abs{\cdot}^{-2})  \]
we have $\dim D(\rmV)^{\und{\chi}}=1$.

When $\und{\chi}=(1,1,1)$, by the same reasoning as in Proposition \ref{case1},  $\dim D(\rmV)^{\und{\chi}}=1$.

When  $\und{\chi}=(\abs{\cdot}, 1, \abs{\cdot}^{-1}   )$. The character $\und{\chi}$ is admissible on $\rmV_f$ and also on $\rmQ^\times$. 
By an easy computation any admissible algebraic character $\rmG$ on $\rmQ^\times$ can be written as 
\[ (a,g, b)\mapsto   a^m \det(g)^n b^k ,\]
where $m,n,k\in \bbZ$ and $m+n=0,n+k=0$.   Hence $\nu $  is not admissible on $\rmQ^\times$.  In view of (\ref{Igusa_ex2}), $Z_{f,|f|^{-1}\mu}$ has pole at $s=0$.  By  Theorem \ref{Main} or Corollary \ref{cor_1},  $Z_{f, |f|^{-1}\mu}$ has simple pole at $s=0$ and $\abs{f}^{-1}\mu$ can not be extended to $\rmV$ as a $(\abs{\cdot},1, \abs{\cdot}^{-1})$-invariant distribution.  
On the other hand,  by automatic extension theorem  (cf.\cite[Theorem 1.4]{HS} )  the $(\abs{\cdot}, 1, \abs{\cdot}^{-1}   )$-invariant distribution supported on $\rmQ^\times$ can be uniquely extended to $\rmV$.   Hence $D(\rmV)^{\und{\chi}}=1$.

When $\und{\chi}=(\abs{\cdot}^2, \abs{\cdot}^{-1},  )$ or  $ (1, \abs{\cdot},  \abs{\cdot}^{-2}) $,  by Lemma \ref{case2lemma} and automatic extension theorem (cf.\cite[Theorem 1.4]{HS} ) we have $D(\rmV)^{\und{\chi}}=1$. 

\epf

In the end, we assume  $n=1$. 

\bl
\label{case3lem}
The only $\rmG$-invariant distribution on $\rmV$ is the delta distribution supported on $\{0\}$.
\el
\bpf
The trivial character $\und{\chi}=(1,1,1)$  is admissible on $\rmF^\times \times \rmF^\times$, $\rmF^\times \times \{0\}$, $\{0\}\times \rmF^\times$ and $\{(0,0)\}$. By Tate thesis, $\und{\chi}$-invariant distribution on $\rmF^\times \times \{0\}$ and $\rmF^\times \times \{0\}$ are not extendable. 

We only need to check the measure $\mu=\frac{dxdy}{\abs{xy}}$ is not extendable.   Assume it can be extended to an invariant distribution $\tilde{\mu}$ on $\rmF\times \rmF$.

By simple computation, the zeta integral $Z_{xy, \frac{dxdy}{\abs{xy}}}$ has pole at $s=0$ of order $2$. Consider the Laurent expansion, 
\[   Z_{xy, \frac{dxdy}{\abs{xy}}}=Z_{-2}s^{-2}+Z_{-1}s^{-1}+Z_0+\cdots . \]
 $\tilde{\mu}-Z_0$ is supported on $\rmF\times \{0\}\cup \{0\}\times \rmF $.  It is clear there exists $\alpha,\beta\in \bbC$ such that
 $\tilde{\mu}-Z_0-\alpha Z^x_{0}-\beta Z^y_{0}$ is supported at $\{(0,0)\}$, where $Z^x_{0}$ (resp. $Z^y_0$) is the 0-th coefficient of the zeta integral $Z_{x,\frac{dx}{\abs{x}}}$ (resp. $Z_{y, \frac{dy}{\abs{y}}}$) on $\rmF\times \{0\}$ (resp.  on $\{0\}\times \rmF$). Hence  $\tilde{\mu}-Z_0-\alpha Z^x_{0}-\beta Z^y_{0}$ is an invariant distribution.  
For any $a,b\in \rmF$ such that $\abs{a},\abs{b}\neq 1$,  we have 
\[ (a-1)(b-1)  (\tilde{\mu}-Z_0-\alpha Z^x_{0}-\beta Z^y_{0})=0 . \]
On the other hand
\begin{align} & (a-1)(b-1)  \cdot (\tilde{\mu}-Z_0- \alpha Z^x_{0}-\beta Z^y_{0}) \\
 &=  (a-1)(b-1) \cdot \tilde{\mu}- (a-1)(b-1)Z_0- (a-1)(b-1)( \alpha Z^x_{0}+\beta Z^y_{0})\\
 &= -\log\abs{a}\log \abs{b}Z_{-2}\neq 0.
\end{align}
 It is a contradiction. It follows that $\frac{dxdy}{\abs{xy}}$ can not be extended to $\rmV$ as an invariant distribution, and the only invariant distribution on $\rmV$ is the delta distribution supported on $\{0\}$.
\epf
\begin{proposition}
\label{case_3}
Assume that $n= 1$.   Then $\dim D(\rmV)^{\und{\chi}}=1$ if and only if  
\[  \chi_1\chi_2\chi_3=1 . \]
  \end{proposition}
\bpf

When  $\chi_1\chi_2\chi_3\neq 1$,  Lemma \ref{basic_lemma_ex2} and Theorem \ref{Localization} implies that $\dim D(\rmV)^{\und{\chi}}=0$.

We now assume that $\chi_1\chi_2\chi_3= 1$.   if $\chi_1\chi_2\neq 1$ or $\chi_2\chi_3\neq 1$, then by automatic extension theorem (cf.\cite[Theorem 1.4]{HS}), we have $\dim D(\rmV)^{\und{\chi}}=1$.

When $\und{\chi}=(1, \chi, \chi^{-1})$ where $\chi$ is not trivial.  
Consider the $\und{\chi}$-invariant distribution  $\chi(y)\frac{dxdy}{|xy|}$.  If $\chi\neq 1$, then the zeta integral $Z_{x, \chi(y)\frac{dxdy}{|xy|}}$ has pole at $s=0$.  The character $\und{\chi}$ is admissible on $V_f$ and $\{0\}\times \rmF^\times$.  Note that the algebraic character  $\nu  (a,g,b)=  ag^{-1} $
is not admissible on $\{0\}\times \rmF^\times$.   By Corollary \ref{cor_1},  $\chi(y)\frac{dxdy}{\abs{xy}}$ is not extendable.  Therefore the only $\und{\chi}$-invariant distribution on $\rmV$ is the distribution $\chi(y)\frac{dy}{\abs{y}}$ supported on $\{0\}\times \rmF^\times$. 

By the same argument, when $\und{\chi}=(\chi, \chi^{-1},1)$ where $\chi\neq 1$,  the distribution  $\chi(x)\frac{dxdy}{\abs{xy}}$  is not extendable as an $\und{\chi}$-invariant distribution on $\rmV$. Therefore the only $\und{\chi}$-invariant distribution on $\rmV$ is the distribution $\chi(x)\frac{dx}{\abs{x}}$ supported on $\rmF^\times \times \{0\}$.

When $\und{\chi}=(1,1,1)$,  the Propotion follows from Lemma \ref{case3lem}.

\epf
%\subsection{Applications to representation theory of p-adic groups}
%
%We define two acton of $\rmF^\times \times \GL_n$ on $V$ one is given by 
%\[(a, g)\cdot_1  (x_1,\cdots, x_n)=(ax_1,\cdots, ax_n)g^{-1},\] 
%and another action is given by 
%\[ (b,g)\cdot_2  (x_1,\cdot, x_n)=(b^{-1}x_1,\cdots,  b^{-1} x_n  )g^t .   \]
%\bt
%For any irreducible smooth representation $\rho$ of $\GL_n$,  we have 
%\[ \Hom_{\GL_n}(\rho, \CS(\rmF^n))\leq 1.  \]
%\et
%\bpf
% We define a bilinear form $(\cdot, \cdot)$ on $\CS(\rmV)$ as follows, 
%\[  (\phi, \psi)=\int_{G}  \phi(g\cdot v)  \psi((g^t)^{-1}\cdot v) \od\!g .  \]
%It satisfies that 
%\[ ((a,g)\cdot_1 \phi, (b,g)\cdot \psi)= \abs{\det(g)}(\phi,  \psi).   \]
%Let $\bar{\rho}$ be the conjugate of $\rho$ in $\CS(\rmV)$ as complex function.   
%Then we get a unique morphisms of representations of $\GL_n$
%\[( , ):  \rho \otimes \bar{\rho} \to  \abs{ \det } \]
%
%Consider the embedding $I_\rho\otimes \bar{\rho} \incl  \CS(\rmF^n)\otimes \CS(\rmF^n)\simeq  \CS(\rmF^n\times \rmF^n)$, where $I_%\rho=\Hom_G(\rho,  \CS(\rmV))\otimes \rho$.
%Applying the functor $\Hom_{F^\times \times \GL_n \times \rmF^n}( \cdot,  \abs{\det} )$ to this embedding, we get the surjective maps
%\[  \Hom_{G}(\CS(\rmF^n\times \rmF^n) ,  \abs{\det}   )  \to  \Hom_G(I_\rho\otimes \bar{\rho},   \abs{\det})=\Hom_G(\rho, \CS(\rmV))^*%\otimes  \Hom_G(\rho \otimes \bar{\rho}, \abs{\det})  .   \]
%
%\epf

\appendix
\section{Extension of intertwining operators and residues (by Shachar Carmeli and Jiuzu Hong)}
\label{appendix}
 In this appendix, we give a homological algebra interpretation of the main result of this paper by relating residues and 1-cocycles.  We work with smooth representations of $\ell$-groups in this appendix.  It would be interesting to see which of the results presented here hold in the archimedean case as well. 

\subsection{Smooth Cohomology of $\ell$-groups}

In this subsection we  recall some basics of smooth cohomology of an $\ell$-group with coefficients in a smooth representation. We  also discuss Ext-groups between smooth representations. For more general theory of smooth cohomology of $\ell$-groups, one can refer to  \cite{BW}. %	However, we only discuss  the non-Archimedean case throughout the whole paper.     

For a smooth representation $(V,\rho)$ of an $\ell$-group $G$, we  recall now the construction of the standard injective resolution of $V$ as a smooth $G$-module. 

Let $C^i(G,V)$ denote the linear space of functions $\phi: G^{i + 1} \to V$ which are locally constant 
on $G^{i + 1}$. The group $G$ acts on $C^i(G,V)$ by 
\[(g\cdot \phi)(g_0,...,g_i) := \rho(g)\cdot \phi(g^{-1}g_0,...,g^{-1}g_i).\] 
There is a natural differential $d : C^i(G,V) \to C^{i + 1}(G,V)$ defined by 
\[ (d\phi) (g_0,...,g_{i + 1}) = \sum_{j = 0}^{i + 1} (-1)^j \phi(g_0,...,\hat{g_j},...,g_{i + 1}).\] 
This way $C^\bullet(G,V)$ becomes a complex.   Let $C^i(G,V)^\infty$ be the space of smooth vectors of $C^i(G,V)$ with respect to the action of $G$, i.e. the set of elements $\phi \in C^i(G,V)$ for which there exists an open  subgroup $K$ of $G$ such that 
\[ \phi(kg_0,kg_1,\cdots, kg_i)=\rho(k)\cdot \phi(g_0,g_1,\cdots, g_i),     \]
for any $k\in K$ and $g_0,g_1,\cdots, g_i\in G$.

From now on we will identify a representation $(V,\rho)$ with its representation space, and denote 
$\rho(g) v$ simply by $g v$. If $V,U$ are two representations, we will always assume that the action of 
$G$ on ${\rm Hom}(V,U)$ is given by $(g\cdot f)(v) = gf(g^{-1}v)$.

\begin{lemma}
For any $i\geq 0$, $C^i(G,V)^\infty$ is an injective smooth representation of $G$.
\end{lemma}
\begin{proof}
It is easy to check that the following linear map 

\[ \begin{array}{rcl}
 C^0(G,W)^\infty   &\rightarrow&   {\rm Ind}_{\{e\}}^G( W|_{\{e\}} )  ,\\
   \phi   &\mapsto&(g\mapsto g^{-1}\cdot  \phi(g) )
   \end{array}\]
   is a well-defined isomorphism of smooth representations of $G$ for any smooth representation $W$ of $G$, where ${\rm Ind}$ is the smooth induction functor.  
   
   Observe that   $C^i(G,V)^\infty$ is naturally isomorphic to $C^0(G, C^{i-1} (G,V) )^\infty$, where by convention $C^{-1}(G,V):=V$.
   Hence  $C^i(G,V)^\infty$ is isomorphic to ${\rm Ind}^G_{\{e\}} (C^{i-1}(G,V)  ) $ as smooth representations of $G$.  Frobenius reciprocity for the smooth induction ${\rm Ind}$ then implies that  $C^i(G,V)^\infty$ is an injective representation of $G$.  
\end{proof}

Let $C^\bullet(G,V)^\infty$ denote the sub-complex $0 \to  C^0(G,V)^\infty\to  \cdots  \to C^i(G,V)^\infty\to \cdots  $ of $C^\bullet(G,V)$.  It is  standard to check that  $V\to C^\bullet(G,V)^\infty$ is an
injective resolution of $V$, where the map $V\to C^0(G,V)^\infty$ is  given by $v\mapsto  \{g\mapsto  g v  \}$.   Thus, by the definition of $Ext^i$, for any two smooth representations $U,V$ of $G$ we have 
\begin{equation}
\label{cohomology formula}
{\rm Ext}^i_G(U,V) \cong H^i({\rm Hom}_G(U,C^\bullet(G,V)^\infty )).      
\end{equation}

Here ${\rm Ext}^i_G$ stands for Ext-groups in the category of smooth $G$-representations.  
One can check easily that the space ${\rm Hom}_G(U,C^i(G,V)^\infty)$ can be identified with the space $C_G^i(U,V)$ consisting  of all maps   $\phi: G^{i+1}\times U\to V$ 
with the following properties:
\begin{itemize}
\item $\phi$ is linear in $U$.
\item  $\phi$ is $G$-equivariant, i.e. for any $g, g_0,\cdots, g_i\in G$, 
\[ \phi(gg_0,gg_1,\cdots, gg_i,g\cdot u)=g\phi(g_0,g_1,\cdots, g_i, u).   \]
\item  For any $u \in U$, $\phi(\cdot, u)$ is a locally constant $V$-valued function.  
\end{itemize}
Let $C^i(U,V)$ be the space consisting of all maps $\psi:  G^i\times U\to V$ with the following properties:
\begin{itemize}
\item  $\psi$ is linear in $U$.
\item  for any $u\in U$,  $\psi(\cdot,u)$ is a locally constant $V$-valued function.
\end{itemize}

Any  map $\phi\in C^{i}_G(U,V)$ can be uniquely associated to  $\tilde{\phi}\in C^i(U,V)$ as follows:
\[ \tilde{\phi}(g_1,\cdots, g_i,u):=\phi(e,g_1,\cdots, g_i,u),\quad \text{ for any } g_1,\cdots,g_i\in G, u\in U . \]
Conversely for any   $\psi  \in C^i(U,V)$, we can uniquely associate an element  $\bar{\psi}  \in C^i_G(U,V)$ as follows,
\[ \bar{\psi}(g,g_1,\cdots, g_i,u):=g\cdot\psi(g^{-1}g_1,\cdots, g^{-1}g_i, g^{-1}u  ), \quad  \text{for any }g,g_1,\cdots, g_i\in G, u\in U.   \]

\begin{definition}
	\label{cocycle and coboundary}
Let $U,V$ be two smooth representations $U,V$ of $G$.
\begin{enumerate}
\item  We call a map $\psi  \in  C^1(U,V)$ a 1-cocycle from $U$ to $V$ if 
\[  \psi(g_1g_2,u)=g_1\cdot \psi(g_2,g_1^{-1}(u))+\psi(g_1,u), \quad \text{for any } g_1,g_2\in G, \text{ and } u\in U.\]
\item  We call a map  $\psi \in C^1(U,V)$ a 1-coboundary from $U$ to $V$ if $\psi(g,u)=g \xi(g^{-1}u) - \xi(u)$ for some $\xi\in {\rm Hom}(U,V)$.  
\end{enumerate}
\end{definition}

By definition, a 1-coboudnary is clearly a 1-cocyle.  Let $Z^1(U,V)$ (resp. $B^1(U,V)$) denote the space of all 1-cocycles (resp. 1-coboundaries) from $U$ to $V$.  
\begin{lemma}
\label{lem_cocycle}
The group ${\rm Ext}^1_G(U, V)$ is naturally isomorphic to the quotient 
$$Z^1(U,V)/ B^1(U,V).$$  
\end{lemma}
\begin{proof}

The first cohomology ${\rm Ext}^1_G(U,V)$ of the complex $C^\bullet_G(U,V)$ is computed by the following quotient
\[ \{\phi \in C^1_G(U,V) \,|\,  d\phi = 0\} / \{  \phi\in C^1_G(U,V)  \,|\,  \phi=d\psi, \text{ for some } \psi \in C^0_G(U,V)\}.\] 

For any $\phi\in C^1_G(U,V)$, the condition $d\phi = 0$ exactly corresponds to 1-cocycle condition for $\tilde{\phi}\in C^1(U,V)$ in the sense of Definition \ref{cocycle and coboundary}.  Similarly $\phi=d\psi$ for some $\psi\in C^0_G(U,V)$ corresponds to the 1-coboundary condition for $\tilde{\phi}$.  
Hence  ${\rm Ext}^1_G(U,V)$ is isomorphic to $Z^1(U,V)/ B^1(U,V)$.

\end{proof}

Let $V,W$ be two smooth representations of $G$. Let $U$ be a sub-representation of $V$.  There exists a long exact sequence 
\[ 0\to {\rm Hom}_G(V/U, W)\to {\rm Hom}_G(V, W)\to {\rm Hom}_G(U,W)\xrightarrow{\delta} {\rm Ext}^1_G(V/U, W)\to \cdots  .\]
\begin{lemma}
\label{cocycle_connect}
With the same notation as above, for any $G$-homomorphism $\xi: U\to W$,  assume that there exists a linear map $\tilde{\xi}: V\to W$ such that $\tilde{\xi}|_U=\xi$.  Then $\delta(\xi)$ is represented by the 1-cocycle from $V/U$ to $W$ given by 
\[ g\mapsto  g\cdot \tilde{\xi} -\tilde{\xi} \in {\rm Hom}(V/U,W) . \]
\end{lemma}

\begin{proof}
Set $\phi(g,v)=g\cdot \tilde{\xi}(g^{-1}\cdot v)-\tilde{\xi}(v)$.  For every vector $v\in V$ and every $g\in G$,  there exists an open subgroup $K$ of $G$ such that $K$ stabilizes $v$ and it also stabilizes the vector $g\cdot \tilde{\xi}(g^{-1}\cdot v )\in W$. It follows that $\phi(\cdot, v)$ is locally constant, i.e.  $\phi\in C^1(V,U)$.  Note that for any $v\in U$,  $\phi(g,v)=0$, hence $\phi$ descends to an element in $C^1(V/U,W)$, and it is easy to check that it satisfies the 1-cocycle condition. Following the standard construction of connecting homomorphism, this 1-cocycle represents the element $\delta(\xi)\in {\rm Ext}^1_G(V/U,W)$. 
\end{proof}
	
\subsection{Residues and 1-cocycles}

In this subsection we relate residues of meromorphic intertwining operators to 
connecting maps in the long exact sequence for extension spaces. 
To state the result, we need to introduce the notion of a meromorphic intertwining operator 
between representations of an $\ell$-group $G$.
\begin{definition}
Let $D$  be an open subset in $\mathbb{C}$.  A $D$-holomorphic character of $G$ is a map $\varpi: D\times G \to \mathbb{C} ^\times$ such that 
\begin{enumerate}
\item for any $g\in G$,  $\varpi(\cdot, g):D\to \mathbb{C}^\times$ is a holomorphic function;
\item there exists an open subgroup $K\subset G$ such that for any $s\in D$, $\varpi(s,\cdot): G\to \mathbb{C}^\times$ is a group homomorphism and $K$ is contained in the kernel.  
\end{enumerate}
\end{definition}

We call a smooth representation $W$ of $G$ \textbf{admissible} if for every open compact subgroup $K \subseteq G$, the space of invariants $W^K$ is finite dimensional. 
\begin{definition}
\label{meromorphic_intertwining}
Let  $V$ be a smooth representation of $G$ and let  $W$ be an admissible smooth representation of $G$. Let $\varpi $ be a $D$-holomorphic  character of $G$.
A $\varpi$-equivariant meromorphic intertwinig operator from $V$ to $W$ over $D$
is a map $\xi : D\times V \to W$  with the following properties:  
\begin{itemize}

\item For every $g \in G$ and $v \in V$, we have $g\cdot  \xi(s,  v) = \varpi(s,g) \xi(s,g\cdot v)$. 

\item  For any $s\in D$,  $\xi(s, \cdot)$ is a linear operator from $V$ to $W$.  For any $v\in V$ there exists an open subgroup $K$ of $G$ that stabilizes $v$, such that  $\chi$ is trivial on $K$ and the induced map $\xi(\cdot,v): D\to W^K$ is a meromorphic function on $D$ with values in the finite-dimensional vector space $W^K$.  

\item There exists a discrete subset $\Pi\subset D$ such that  $\xi(\cdot, v)$ is holomorphic outside of $\Pi$ for any $v\in V$ and the orders of poles at the points of $\Pi$ are uniformly bounded with respect to $v\in V$ at any point of $\Pi$.  
\end{itemize} 
\end{definition}
 
Given a $\varpi$-equivariant meromorphic intertwining operator $\xi: D\times V\to W$, we denote by $\xi(s)$ the associated linear operator $\xi(s, \cdot)$. Clearly $\xi(s)$ is a $G$-homomorphism from $V$ to the representation 
$W\otimes \varpi(s)$, where $\varpi(s)$ denotes the character  $\varpi(s,\cdot)$ of $G$. We consider the Laurent expansion of $\xi(s)$ at $s=s_0$, 
\[ \xi(s)=\sum_{i=-k_0}^\infty \xi_i\cdot  (s-s_0)^i , \]
where $k_0$ is the order of the pole of $\xi(s)$ at $s=s_0$, and  the coefficient $\xi_i$ is a linear operator from $V$ to $W$ for each $i$. The coefficient $\xi_{-1}$ is called the $\bold{residue}$ of $\xi(s)$ at $s=s_0$, denoted by ${\rm Res}_{s=s_0}\xi(s)$. 

We now state the main result of this appendix. 
	
\begin{theorem}
\label{Main_Theorem}
 Let  $\xi$ be a $\varpi$-equivariant meromorphic intertwining operator from $V$ to $W$ over $D$, where $V$ is a smooth representation and $W$ is a smooth admissible representation of $G$. 
Let $U \subseteq V$ be a sub-representation such that $\xi|_U$ is holomorphic at $s_0 \in D$.   Let $\delta: {\rm Hom}_G(U,W )\to {\rm Ext}^1_G(V/U,W)$ be the connecting homomorphism.  Assume that $\varpi(s_0)$ is the trivial character of $G$. Then
\begin{enumerate}
\item  For any $s_0\in D$,  the map ${\rm Res}_{s=s_0}  (\frac{  \varpi(s)-1 }{ s-s_0  } \xi(s)  )  $ given by
\[ g\mapsto   {\rm Res}_{s=s_0}( \frac{\varpi(g,s)-1}{s-s_0}\xi(s))   \]
is a 1-cocycle from $V$ to $W$.  
\item    The class of $\delta(\xi|_U(s_0))$ in ${\rm Ext}^1_G(V/U,W)$ is represented by the 1-cocycle ${\rm Res}_{s=s_0}  (\frac{  \varpi(s)-1  }{ s-s_0  } \xi(s)  )  $.
\end{enumerate}
\end{theorem}

\begin{proof}
   Recall that $\xi_0$ is the 0-th coefficient in the Laurent expansion of $\xi(s)$ at $s=s_0$. Therefore, $\xi_0:  V\to W$ gives an extension of $\xi|_U(s_0): U\to W$ as a linear operator (not necessarily equivariant). Moreover, it is a smooth vector of ${\rm Hom}(V,W)$ since it is fixed by the kernel of $\varpi$, which is open in $G$.  In view of Lemma \ref{cocycle_connect},  $\delta( \xi|_U(s_0))$ can be represented by the 1-cocycle $g\cdot \xi_0-\xi_0$. Finally, note that $\xi_0={\rm Res}_{s=s_0}\frac{\xi(s) }{s-s_0}$ and hence 
\[  g\cdot \xi_0-\xi_0={\rm Res}_{s=s_0}\frac{g\cdot \xi(s)}{s-s_0}- {\rm Res}_{s=s_0}\frac{\xi(s) }{s-s_0}={\rm Res}_{s=s_0}( \frac{\varpi(g,s)-1}{s-s_0}\xi(s)).\]
This finishes the proof. 

\end{proof}

% The following corollary is immediate. 
%\begin{corollary}
%With the same setting as in Theorem \ref{Main_Theorem}.   If $\xi(s)$ has simple pole at $s_0$, then $\delta(\xi|_U(s_0))$ is represented by 
%\[g\mapsto \frac{d\varpi(g,s)}{ds}(s_0)  {\rm Res}_{s=s_0}  \xi(s).  \]
%\end{corollary}

\subsection{Generalized homomorphisms and a non-vanishing criterion of  residue 1-cocycles}

%In this sub-section we use the formula from Theorem \ref{Main_Theorem} to deduce the main result of the main body of this paper once again. Namely, we show how the assumptions of Theorem 
%\ref{main_thm} allow one to show that the connecting homomorphism attached to the sequence 
%$0 \to \m{S}(X_f) \to \m{S}(X) \to \m{S}(V(f)) \to 0$ by applying $Hom_G(-,)$,
%does not vanish on the semi-invariant distribution supported on the open orbit $X_f$.  

Let $V,W$ be two smooth representations of $G$. We first recall  the definition of generalized $G$-homomorphisms defined in [HS]. The group $G$ acts on ${\rm Hom(V,W)}$ naturally. The space ${\rm Hom}_{G,k}(V,W)$ of generalized $G$-homomorphisms from $V$ to $W$ of order $\leq k$ consists of $\xi\in {\rm Hom}(V,W)$ such that
\[  (g_0-1)(g_1-1)\cdots (g_k-1)\cdot \xi=0 , \quad \text{ for any } g_0,\cdots, g_k\in G.  \]   

The space ${\rm Hom}_{G,\infty}(V,W) $ of all generalized $G$-homomorphisms is the union 
\[ {\rm Hom}_{G,\infty}(V,W):= \bigcup_{k=0}^\infty {\rm Hom}_{G, k}(V,W) .\]

Assume that $W$ is an admissible representation of $G$.  Let $\xi(s)$ be a $\varpi$-equivariant meromorphic intertwining operator from $V$ to $W$, where $\varpi$ is a holomorphic character of $G$.  Recall the Laurent expansion of $\xi(s)$ at $s=s_0$, 
\[ \xi(s)=\sum_{i=-k_0}^\infty  \xi_i  (s-s_0)^i ,\]
where $k_0$ is the order of the pole of $\xi(s)$ at $s=s_0$. Assume further that $\varpi(g,s_0)=1$ for any $g\in G$.  Consider the Taylor expansion of $\varpi(g,s)$ at $s=s_0$. 
\[ \varpi(g,s)=1+ \sum_{i=1}^\infty \varpi_i(g,s_0) (s-s_0)^ i . \]

\begin{lemma}
\label{formula_coeff}
With the same notation as above,  we have 
\[ (g-1)\xi_i=\sum_{j=1}^{i+k_0} \varpi_j(g,s_0)\xi_{i-j}\]
and in particular
$\xi_i \in {\rm Hom}_{G, i+k_0}(V,W) $ for each $i$. 
\end{lemma}
\begin{proof}
The lemma  follows from the following equivariant property:
\[ g\cdot \xi(s)=\varpi(g,s)\xi(s).  \]
\end{proof}

By Theorem \ref{Main_Theorem} and Lemma \ref{formula_coeff}, we immediately have the following corollary: 

\begin{corollary}
The class $\delta(\xi|_U(s_0))$ is represented by 
\[ g\mapsto  \sum_{i=1}^{k_0} \varpi_i(g,s_0) \xi_{-i} .\]
\end{corollary}

We deduce from this a criterion for the non-vanishing of $\delta(\xi|_U(s_0))$.  
\begin{theorem}
\label{Criterion_Non_Van}
Assume that $k_0\geq 1$, and  there exists $g_1,g_2,\cdots, g_{k_0}\in G$  such that
\begin{enumerate}
\item  $\frac{d \varpi(g_i,s)}{ds}(s_0)\neq 0$ for each $i$.
\item   $(g_1-1)(g_2-1)\cdots (g_{k_0}-1)$ acts on ${\rm Hom}_{G,\infty}(V/U,W) $ by zero. 
\end{enumerate}
Then $\delta(\xi|_U(s_0))\neq 0$.  
\end{theorem}
\begin{proof}
Note that $\varpi_1(g,s_0)=\frac{ d\varpi(g,s)}{ds}(s_0)$.  By repeating Lemma \ref{formula_coeff}, we have 
\[ (g_1-1)(g_2-1)\cdots (g_{k_0}-1)\cdot \xi_0=(\prod_{i=1}^k \frac{d\varpi(g_i,s)}{ds}(s_0))  \xi_{-k_0} . \]

Assume that $(g\mapsto g\cdot \xi_0-\xi_0)$ is  a 1-coboundary from $V/U$ to $W$, then there exists  $\bar{\xi} \in {\rm Hom}(V/U,W)$ such that $g\cdot \xi_0-\xi_0=g\cdot \bar{\xi}-\bar{\xi}$ for any $g\in G$. By Lemma \ref{formula_coeff} $g\cdot \xi_0-\xi_0$ is a generalized $G$-homomorphism,  it follows that $\bar{\xi}\in {\rm Hom}_{G,\infty}(V/U,W)$.  By assumption (2),  
\[ (g_1-1)(g_2-1)\cdots (g_{k_0}-1)\cdot \bar{\xi}=0.  \]
But the left-hand side equals to 
\[  (g_1-1)\cdots  (g_{k_0-1} -1)(g_{k_0}-1)\cdot \xi_0 \] by the equality 
$g\tilde{\xi} - \tilde{\xi} = g \xi_0 - \xi_0$. 
The distribution $(g_1-1)\cdots  (g_{k_0-1} -1)(g_{k_0}-1)\cdot \xi_0$ is nonzero since $\frac{d \varpi(g_i,s)}{ds}(s_0)\neq 0$ for each $i$.  We arrive at a contradiction, and hence $\delta(\xi|_U(s_0))\neq 0$.

\end{proof}

%\subsection{Application to Igusa Zeta Integrals}
%Let $X$ be an algebraic variety over a non-archimeadian local field  $F$ of characteristic zero. Let $G$ be a connected algebraic group acting on $X$.  Assume $f$ is a $\nu$-invariant regular function on $X$ where $\nu: G\to F^\times$ is an algebraic character of $G$. Let $X_f$ denote the open subset of $X$ consisting of points on which $f$ does not vanish, and let $V(f)$ denote the complement set of $X_f$ in $X$.  
% Let $\mu$ be a $\chi$-invariant measure  on  $X_f$. 
We shall now come back to the setup of Section \ref{main_thm}. Let $X$ be the space $\sfX(F)$ and and $G=\sfG(F)$.
%We attach the Igusa zeta integral  $Z_{f,\mu}$ on $X$.
 Assume that $Z_{f,\mu}$ is standard at $s=s_0$ in the sense of Section \ref{main_thm}.    It  amounts to saying that $Z_{f,\mu}$ is a  $|\nu|^s$-equivariant meromorphic intertwining operator from $S(X)$ to $\chi$.  It is interesting to understand when $\mu$ can be extended to a $\chi$-invariant distribution on $X$.  It is equivalent to the vanishing of the class $\delta(\mu)$ in ${\rm Ext}^1_G(S(X_f), \chi)$. In Theorem \ref{Main},  the  conditions for the unextendability of $\mu$ as semi-invariant distribution to $X$ are exactly to make sure that assumption (2) in Theorem  \ref{Criterion_Non_Van} holds.  The main body of this paper is exactly to verify the assumption (2). 
 Up to this technicality,  we essentially re-proved Theorem \ref{Main} from the point of view of homological algebra. 
   
{\bf Acknowledgements}: 
We would like to thank D.\,Gourevitch for his careful reading and comments.  
S.\,Carmeli is partially supported by the Adams Fellowship of the Israeli Academy of Science,
and partially supported by ERC StG grant 637912, and ISF grant 249/17.

 \end{document}